\documentclass[10pt]{article}

\usepackage[hidelinks]{hyperref}
\usepackage{caption}
\usepackage{xcolor}
\hypersetup{
    colorlinks,
    linkcolor={red!50!black},
    citecolor={blue!50!black},
    urlcolor={blue!80!black}
}
\usepackage{enumitem}
\usepackage[margin=1in]{geometry} 
\usepackage{changepage}
\usepackage{pdflscape}
\usepackage{lineno}

\newcommand{\tcr}{\textcolor{red}}
\usepackage{floatrow}
\usepackage[font=small,labelformat=simple]{caption, subfig}
\usepackage{graphbox}
\usepackage{adjustbox}

\usepackage{longtable}
\usepackage{booktabs}
\usepackage{array}
\usepackage{enumitem}
\usepackage{appendix}
\usepackage{titling}
\usepackage{enumitem}

\usepackage{tikz}
\usepackage{forest}
\usepackage{array}
\usepackage{adjustbox}

\forestset{
    default preamble={
        where n children=0{tier=leaf}{},
        for tree={
            calign=last,
            l sep-=.7em,l=0,
            s sep-=0.5em,
            parent anchor=center, child anchor=center,
        }
    }
}
\newsavebox{\treebox}
\newcolumntype{T}[1]{>{\begin{lrbox}{\treebox}}m{#1}<{\end{lrbox}\centering\begin{adjustbox}{max width=#1}\unhbox\treebox\end{adjustbox}}}

\usepackage{mathtools} 
\usepackage{amssymb}
\usepackage[amsmath, amsthm, thmmarks]{ntheorem}
\usepackage[capitalize, noabbrev]{cleveref}
\crefname{equation}{Eq.}{Eqs.}

\newtheorem{theorem}{Theorem}[section]
\newtheorem{proposition}[theorem]{Proposition}
\newtheorem{corollary}[theorem]{Corollary}
\newtheorem{lemma}[theorem]{Lemma}
\newtheorem{definition}[theorem]{Definition}
\newtheorem{conjecture}[theorem]{Conjecture}

\theorembodyfont{\upshape}
\theoremstyle{break}

\usepackage{natbib}

\renewcommand{\geq}{\geqslant}
\renewcommand{\leq}{\leqslant}

\renewenvironment{proof}{\noindent {\bfseries Proof.}}{$\square$}

\title{Extremal ranks of unlabeled multifurcating rooted trees \\ in a bijective encoding by the positive integers}
\author{Michael R Doboli$^*$, Alessandra R P Maranca$^*$, 
and Noah A Rosenberg\thanks{Department of Biology, Stanford University, Stanford, CA 94305, USA}}
\date{\today}

\begin{document}
\maketitle

\noindent {\bf Abstract.} Maranca \& Rosenberg~\cite{Maranca2024} devised a ranking scheme for unlabeled multifurcating rooted trees, in which the trees are bijectively associated with the positive integers. Here, generalizing earlier results for bifurcating trees, we determine, for trees with a fixed number of leaves, which multifurcating trees obtain the maximal and minimal ranks. We identify these maximizing and minimizing trees for each of two sets of unlabeled multifurcating rooted trees: strictly $k$-furcating trees, in which each internal node possesses exactly $k$ descendants, and at-most-$k$-furcating trees, in which internal nodes possess at least 2 and at most $k$ descendants. In both scenarios, we find that a tree that can be regarded as maximally balanced attains the minimal rank, and a minimally balanced tree attains the maximal rank. We deduce recurrences for the maximal and minimal rank for trees with fixed numbers of leaves in both the strictly $k$-furcating and at-most-$k$-furcating cases. The maximal rank on $(n-1)(k-1)+1$ leaves grows with $(k!)^{\frac{1}{k-1}} \beta_k^{(k^n)}$ in the strictly $k$-furcating case, and the maximal rank on $n$ leaves grows with $(k!)^{\frac{1}{k-1}} \gamma_k^{(k^n)}$ in the at-most-$k$-furcating case, where $\beta_k > 1$ and $\gamma_k > 1$ are constants that depend on the value of $k$. We show that $\beta_k$ decreases as the value of $k$ increases, and that $\gamma_k > \beta_k$ for $k \geq 3$. The results contribute to the use of tree encodings for empirical characterization of phylogenies and measurement of tree balance.
 
\vskip .3cm
{\small
\noindent {\bf Mathematics subject classification (2020):} 05C05, 05C30, 92D15 }

\vskip .15cm
{\small
\noindent {\bf Keywords:} mathematical phylogenetics, recurrences, unlabeled trees 
}

\section{Introduction}

Encodings are a central topic of mathematical phylogenetics~\cite{DressEtAl12, SempleAndSteel03}. Given a class of trees, to what extent can a tree from the class be encoded by features such as quartets, splits, or pairwise distances between leaves---or by a single positive integer? 

For the set of unlabeled rooted binary trees, a scheme introduced by Colijn \& Plazzotta~\cite{Colijn2018} bijectively associates positive integers with trees. The single-leaf tree is assigned rank 1. Each subsequent tree is then associated with an ordered pair $(j_1, j_2)$, $j_1 \geq j_2 \geq 1$, where $j_1$ is the rank of the left subtree and $j_2$ is the rank of the right subtree; ``left'' and ``right'' are used for convenience, but trees are non-plane. The rank of the tree is then computed by finding the position of $(j_1,j_2)$ in the dictionary ordering of ordered pairs $(a,b)$ satisfying $a \geq b \geq 1$: $(1,1), (2,1), (2,2), (3,1), (3,2)$, and so on. For example, the tree with rank 3 is the tree associated with ordered pair $(2,1)$, or the caterpillar tree on 3 leaves. More generally, we can directly compute that the tree associated with $(j_1,j_2)$ has rank $j_1(j_1-1)/2 + 1 + j_2$. This map between trees and the positive integers is in fact bijective. Examples of the ranks of small bifurcating trees appear in Table~\ref{table:bifurcating_CP_rank}.

In a further analysis of this bijection, Rosenberg~\cite{Rosenberg2021} identified the trees with the smallest and largest rank among all trees with a given number of leaves. The tree with the smallest rank on $n$ leaves is a specific highly balanced tree, with rank $a_n$ satisfying $a_n = \frac{1}{2} a_{\lceil n / 2 \rceil} (a_{\lceil n/2 \rceil} - 1) + 1 + a_{\lfloor n / 2 \rfloor}$, $a_1 = 1$. The tree with the maximal rank is the caterpillar, the minimally balanced tree, with rank $b_n$ satisfying $b_n = \frac{1}{2} b_{n-1}(b_{n-1}-1) + 2$, $b_1 = 1$. A curious relation $a_{2^n} + 1 = b_{n+2}$ holds between the minimal and maximal ranks for all $n \geq 0$. 

Rosenberg~\cite{Rosenberg2021} also studied the asymptotic growth of the smallest and largest ranks among $n$-leaf trees, showing $b_n \sim 2 \beta^{(2^n)}$, where $\beta \approx 1.05653$. An immediate corollary is $a_{2^n} \sim 2 \alpha^{(2^n)}$, where $\alpha = \beta^4 \approx 1.24602$. Doboli et al.~\cite{DoboliEtAl24} then found $a_n \sim 2[2^{P(\log_2 n)}]^n$, where $P$ is a 1-periodic function satisfying $1.24602 < 2^{P(t)} < 1.33429$ for all $t$. This result improved upon an earlier upper bound $a_n < ( \frac{3}{2} )^n$ of Rosenberg~\cite{Rosenberg2021}.

Maranca \& Rosenberg~\cite{Maranca2024} have generalized the ranking scheme of Colijn \& Plazzotta~\cite{Colijn2018} to provide a scheme for ranking unlabeled rooted \emph{multifurcating} trees. In the Maranca--Rosenberg scheme for \emph{strictly $k$-furcating} trees, where each internal node has exactly $k$ descendants, we associate to each tree a $k$-tuple $(j_1,j_2,\ldots,j_k)$, where $j_1 \geq j_2 \geq \ldots \geq j_k \geq 1$ and $j_i$ denotes the rank of the $i$th subtree (canonically ordered from left to right). Each tuple is then associated to a positive integer corresponding to its position in the dictionary ordering of $k$-tuples $(j_1, j_2, \ldots, j_k)$, where $j_1 \geq j_2 \geq \ldots \geq j_k \geq 1$. For example, with $k=3$, after rank 1 is assigned to the 1-leaf tree, the dictionary ordering begins $(1,1,1), (2,1,1), (2,2,1), (2,2,2), (3,1,1)$. This generalized scheme provides a bijection between the strictly $k$-furcating trees and the positive integers.

Maranca \& Rosenberg~\cite{Maranca2024} also provided a ranking scheme for \emph{at-most-$k$-furcating} trees, where each internal node has at least two and at most $k$ descendants. To accommodate the fact that the root might have fewer than $k$ immediate descendant nodes, each tree is associated with a tuple $(j_1,j_2, \ldots, j_k)$, where $j_1 \geq j_2 \geq \ldots \geq j_k \geq 0$ and $j_1 \geq j_2 \geq 1$. The latter condition encodes the fact that each internal node must possess at least two descendant nodes. Each tuple is then associated to a positive integer corresponding to its position in the dictionary ordering of $k$-tuples $(j_1, j_2, \ldots, j_k)$, where $j_1 \geq j_2, \ldots \geq j_k \geq 0$ and $j_1 \geq j_2 \geq 1$. For example, for $k=3$, after the initial 1-leaf tree, the dictionary ordering of these tuples begins $(1,1,0), (1,1,1), (2,1,0), (2,1,1), (2,2,0), (2,2,1), (2,2,2), (3,1,0)$. This labeling scheme provides a bijection between at-most-$k$-furcating trees and positive integers. 

In this study, we conduct analogous work to Rosenberg~\cite{Rosenberg2021} on the maximal and minimal rank, but for the multifurcation schemes of Maranca \& Rosenberg~\cite{Maranca2024}. We first reframe the schemes of Maranca \& Rosenberg~\cite{Maranca2024} by providing a notion of lexicographically ordering tuples. We use this notion to prove a useful relation between two trees' ranks and the ranks of their subtrees (\cref{sec:prelims}). Next, we find the strictly $k$-furcating trees with maximal (\cref{subsec:strictly_k_maximal_rank}) and minimal rank (\cref{subsec:strictly_k_minimal_rank}). We also compute a recursive formula describing the maximal rank of strictly $k$-furcating trees on $(n-1)(k-1)+1$ leaves, $b_n$, and the minimal rank of strictly $k$-furcating trees on $(n-1)(k-1)+1$ leaves, $a_n$. We explore the asymptotic growth of $a_n$ and $b_n$ (\cref{subsec:strictly_k_asymptotics}). We then address these same questions in the at-most-$k$-furcating case (\cref{sec:at_most_k_furcating}), arriving at similar answers. We conclude with a discussion (\cref{sec:discussion}).

\begin{table}[tb]
\begin{tabular}{|c|c|T{2cm}|} 
\hline
$f(t)$ & $K(t)$ & $t$ \\
\hline
1 & $(1, 0)$ & \begin{forest}[[]]\end{forest} \\
2 & $(1,1)$ & \begin{forest}[[][]]\end{forest}\\
3 & $(2,1)$ & \begin{forest}[[[][]][]]\end{forest} \\
4 & $(2,2)$ & \begin{forest}[[[][]][[][]]]\end{forest} \\
5 & $(3,1)$ & \begin{forest}[[[[][]][]][]]\end{forest} 
\\
6 & $(3,2)$ & \begin{forest}[[[[][]][]][[][]]]\end{forest} \\
7 & $(3,3)$ & \begin{forest}[[[[][]][]][[[][]][]]]\end{forest} \\
8 & $(4,1)$ & \begin{forest} [[[[][]][[][]]] []] \end{forest} \\
9 & $(4,2)$ & \begin{forest}[[[[][]][[][]]][[][]]]\end{forest} \\
\hline
\end{tabular}
\vspace{-.2cm}
\captionof{table}{The first several ranks $f(t)$ of unlabeled rooted binary trees. $K(t)$ values are sorted in lexicographically increasing order (after the trivial first row). The left subtree of tree $t$ has rank equal to the first number in $K(t)$, and the right subtree has rank equal to the second number in $K(t)$. The tree corresponding to rank 1 is defined to be the tree with one leaf, and we write $K(t)=(1,0)$ for this tree.}
\label{table:bifurcating_CP_rank}
\end{table}

\section{Preliminaries} 
\label{sec:prelims}

\subsection{Lexicographical ordering}

We first recall a notion of lexicographically ordering $k$-tuples of non-negative integers. 
\begin{definition}
Let $A_k = \{(a_1,a_2,\ldots,a_k) \in \mathbb{Z}^k: a_1 \geq a_2 \ldots \geq a_k \geq 0 \} \subset \mathbb{Z}^k$ for $k \geq 1$. For $X,Y \in A_k$ with $X=(x_1,x_2,\ldots,x_k)$ and $Y=(y_1,y_2,\ldots,y_k)$, we say that $X \leq_D Y$ if either of the following holds: 
\begin{enumerate}[label=(\roman*)]

\item $x_i = y_i$ for all $i$, $1 \leq i \leq k$.

\item $x_i \neq y_i$ for some $i$, and if $j$ is the smallest positive integer such that $x_j \neq y_j$, then $x_j < y_j$.

\end{enumerate}
\end{definition}

\noindent If $X,Y \in A_k$ satisfy $X \leq_D Y$ and $X \neq Y$, then we can also write $X <_D Y$.

\subsection{Strictly $k$-furcating trees} 
\label{subsec:prelims-strictly-k}

For $k \geq 2$, we use the ranking scheme of Maranca \& Rosenberg~\cite{Maranca2024} for strictly $k$-furcating trees. Denote $T_k$ as the set of strictly $k$-furcating trees, with the single-leaf tree also included. We denote $f:T_k \to \mathbb{Z}^+$ as the map sending each strictly $k$-furcating tree to its \emph{rank}.

For all trees $t$ in $T_k$ other than the single-leaf tree, the \textit{canonical order} of $t$ is the ordering of subtrees $t_1, t_2, \ldots,t_k$ at the root such that $f(t_1) \geq f(t_2) \geq \ldots \geq f(t_k)$. In a planar representation of $t$, it is convenient to represent the canonical order from left to right, with $t_1$ as the leftmost and $t_k$ as the rightmost subtree.

For a given tree $t \in T_k$ (other than the single-leaf tree), define $K(t) = \big(f(t_1), f(t_2), \ldots, f(t_k)\big)$, where $t_1, t_2, \ldots, t_k$ is the canonical order of the subtrees of $t$ at the root. By definition of the canonical order, $f(t_1) \geq f(t_2) \ldots \geq f(t_k)$. 

Definition 4.2 of Maranca \& Rosenberg~\cite{Maranca2024} establishes that if $K(t) = (x_1, x_2, \ldots, x_k)$ and $t \in T_k$, then: 
\begin{align} 
\label{eq:rank_strictly_k}
f(t) = 2 + \sum_{i=1}^{k} \binom{x_{k-i+1} + i -2}{i}.
\end{align}
We prove a lemma relating the rank of a strictly $k$-furcating tree with the lexicographical order of the canonical ordering. The proof is in Appendix~\ref{pf:strict_tree_order}. This lemma confirms the intuition that the ordering of trees in the canonical ordering is the same as the ordering induced by the tree rank. It is used extensively in~\cref{subsec:strictly_k_maximal_rank} and~\cref{subsec:strictly_k_minimal_rank} to compare the ranks of trees by instead comparing the lexicographical ordering of the trees' canonical orderings.

\begin{lemma} 
\label{lemma:strict_tree_order}
For two trees $t_1,t_2 \in T_k$, $K(t_1) \leq_D K(t_2)$ if and only if $f(t_1) \leq f(t_2)$. Additionally, $K(t_1) = K(t_2)$ if and only if $t_1 = t_2$.
\end{lemma}

\subsection{At-most-$k$-furcating trees}

We extend our results from~\cref{subsec:prelims-strictly-k} to at-most-$k$-furcating trees (including the single-leaf tree). Let $T_k^{\ast}$ be the set of at-most-$k$-furcating trees. We denote $f:T_k^{\ast} \to \mathbb{Z}^+$ as the map sending each at-most-$k$-furcating tree to its rank, as in Definition 5.2 of~\cite{Maranca2024}. By Theorem 5.3 of~\cite{Maranca2024}, this map is a bijection. Note that we use the same symbol $f$ as in the map $f:T_k \to \mathbb{Z}^+$ in~\cref{subsec:prelims-strictly-k}; the version of $f$ will be clear from the context. 

For all trees $t$ in $T_k^{\ast}$ other than the single-leaf tree, we define the \textit{canonical order} of $t$ to be the ordering of subtrees $t_1, t_2, \dots, t_k$ at the root such that $f(t_1) \geq f(t_2) \geq \dots \geq f(t_k)$. For any tree $t \in T_k^{\ast}, $ define $K(t) = \big(f(t_1), f(t_2), \ldots, f(t_k)\big)$, where $t_1, t_2, \ldots, t_k$ is the canonical order of the subtrees of $t$ at the root. By definition of the canonical order, $f(t_1) \geq f(t_2) \ldots \geq f(t_k)$. The key difference in the at-most-$k$-furcating case compared to the strictly $k$-furcating case in~\cref{subsec:prelims-strictly-k} is that the at-most-$k$-furcating case allows $f(t_k) = 0$ (i.e., if the subtree $t_k$ has no leaves) for $k \geq 3$. 

By Definition 5.2 of Maranca \& Rosenberg~\cite{Maranca2024}, if $t \in T_k^{\ast}$ and $K(t) = (x_1, x_2, \ldots, x_k)$, then 
\begin{align} 
\label{eq:rank_at_most_k}
f(t) = -x_1 + 1 + \sum_{i=1}^{k} \binom{x_{k-i+1} + i -1}{i}.
\end{align}
We show an analogous result to \cref{lemma:strict_tree_order} for at-most-$k$-furcating trees. The proof is in Appendix~\ref{pf:up_to_k_tree_order}. This lemma confirms that the lexicographic ordering induced by canonical ordering is the same as the ordering induced by the tree rank, and we use it in~\cref{subsec:at_most_k_maximal_rank} and~\cref{subsec:at_most_k_minimal_rank} to compare tree ranks.

\begin{lemma} 
\label{lemma:up_to_k_tree_order}
For two trees $t_1,t_2 \in T_k^{\ast}$, $K(t_1) \leq_D K(t_2)$ if and only if 
$f(t_1) \leq f(t_2)$. Additionally, $K(t_1) = K(t_2)$ if and only if $f(t_1) = f(t_2)$.
\end{lemma}

\section{Strictly $k$-furcating trees} 
\label{sec:strictly_k}

\subsection{Maximal rank} 
\label{subsec:strictly_k_maximal_rank}

We now identify the strictly $k$-furcating tree that attains maximal rank among trees with a fixed number of leaves and find a recurrence to compute this rank. Unlike for strictly bifurcating trees ($k=2$), for $k \geq 3$, not every positive integer can be the number of leaves of a strictly $k$-furcating tree. 

Denote by $m:T_k \to \mathbb{Z}$ the number of leaves possessed by a strictly $k$-furcating tree. The number of leaves for a tree $t \in T_k$ satisfies $m(t) \equiv 1 \pmod{k-1}$. This result can be seen inductively. A strictly $k$-furcating tree is formed by successive application of a step of replacing a leaf by an internal node with $k$ descendant leaves. This step increases the number of leaves by $k-1$, so that beginning with a single leaf, each step produces a number of leaves congruent to 1 modulo $(k-1)$.

Because the number of leaves possessed by a tree in $T_k$ is congruent to 1 modulo $(k-1)$, we modify the notation from~\cite{Rosenberg2021} for the minimum and maximum. In particular, for $n \geq 1$, we denote by $a_n$ and $b_n$ the minimal and maximal rank, respectively, across trees in $T_k$ with $(n-1)(k-1)+1$ leaves. Similarly, we denote by $z_n$ the tree in $T_k$ with rank $a_n$ and by $Z_n$ the tree in $T_k$ with rank $b_n$. This notation accords with~\cite{Rosenberg2021} for the case of $k=2$.

The strategy is to construct a collection of trees $\{Z_n^{\ast}\}$ and to show via induction that $Z_n = Z_n^{\ast}$. Let $\{Z_n^{\ast}\}_{n\geq 1}$ be the collection of trees defined as follows: 
(i) $Z_1^{\ast}$ is the tree with exactly one leaf; 
(ii) for $n > 1$, $Z_n^{\ast}$ is the tree that has $Z_{n-1}^{\ast}, Z_1^{\ast}, \ldots, Z_1^{\ast}$ as its $k$ subtrees descended from the root. Note that $Z_{n-1}^{\ast}$ has $(n-2)(k-1)+1$ leaves, and $Z_{n}^{\ast}$ has $(n-1)(k-1) + 1$ leaves.

We show that the rank $f(Z_n^{\ast})$ is strictly increasing with $n$. 
\begin{proposition} 
\label{prop:strict_monotonic}
$f(Z_n^{\ast}) > f(Z_{n-1}^{\ast})$ for all $n \geq 2$.
\end{proposition}
\begin{proof}
By the recursive definition of $Z_n^{\ast},$ we know that $K(Z_n^{\ast}) = 
\big(f(Z_{n-1}^{\ast}), 1, \ldots, 1\big)$. We have $f(Z_1^{\ast}) = 1$, and by \cref{eq:rank_strictly_k}, for $n > 1$, 
\begin{align*}
f(Z_n^{\ast}) &= 2 + \binom{f(Z_{n-1}^{\ast}) + k - 2}{k} + \sum_{i=1}^{k-1} \binom{1 + i - 2}{i} 
=2 + \binom{f(Z_{n-1}^{\ast}) + k - 2}{k}.
\end{align*}
In Appendix~\ref{sec:proof_binomial_ineq}, we prove that $2 + \binom{x + k - 2}{k} > x$ for all nonnegative integers $x$, implying the desired inequality $f(Z_n^{\ast}) = 2 + \binom{f(Z_{n-1}^{\ast}) + k - 2}{k} > f(Z_{n-1}^{\ast})$.
\end{proof}

\medskip
We use Proposition~\ref{prop:strict_monotonic} and the recursive construction of $Z_n^{\ast}$ to show that $Z_n = Z_n^{\ast}$, demonstrating that the tree with maximal rank among strictly $k$-furcating trees with $(n-1)(k-1) + 1$ leaves is exactly $Z_n^{\ast}.$
\begin{theorem}
For $n \geq 1$, the strictly $k$-furcating tree $Z_n$ with maximal rank is $Z_n=Z_n^{\ast}$. 
\label{thm:strict_max_equals_ast}
\end{theorem} 
\begin{proof}
We induct on $n$. The base case $n=1$ is trivial, as only one tree has a single leaf: $Z_1 = Z_1^{\ast}$. 

For the inductive hypothesis, suppose that $Z_\ell = Z_\ell^{\ast}$ for each $\ell$ with $1 \leq \ell \leq n-1$. In the tree $Z_n$, let $j_1, j_2, \ldots, j_k$ be the subtrees, in canonical order. The idea is to substitute the leftmost tree $j_1$ with the tree $Z_{[(m(j_1) + k - 2]/(k - 1)}$ that attains the maximal rank on $m(j_1)$ leaves. Trees $j_1$ and $Z_{[(m(j_1) + k - 2]/(k - 1)}$ have the same number of leaves, because $m(Z_{[(m(j_1) + k - 2]/(k - 1)}) = (k-1)(\frac{m(j_1) + k - 2}{k - 1}-1)+1 = m(j_1)$. The substitution does not break the canonical order, because the rank of $Z_{[(m(j_1) + k - 2]/(k - 1)}$ is at least as large as that of $j_1$. It produces the following chain of inequalities:
\begin{align*}
K(Z_n) = \Big(f(j_1), f(j_2), \ldots, f(j_k)\Big) & \leq_D \Big(f(Z_{\frac{m(j_1)+k-2}{k-1}}), f(j_2), \ldots, f(j_k)\Big) \\
& = \Big( f(Z^{\ast}_{\frac{m(j_1)+k-2}{k-1}}), f(j_2), \ldots, f(j_k)\Big) \\ 
& \leq_D \Big( f(Z^{\ast}_{n-1}), f(Z^{\ast}_{1}), \ldots, f(Z^{\ast}_{1}) \Big) \\
& = K(Z_n^{\ast}).
\end{align*}

The second equality follows from the inductive hypothesis. The last inequality follows from~\cref{prop:strict_monotonic} and the fact that $\frac{m(j_1) + k -2}{k-1} \leq n-1$. By definition of $Z_n$, $f(Z_n) \geq f(Z_n^{\ast})$, so by~\cref{lemma:strict_tree_order}, $K(Z_n) \geq_D K(Z_n^{\ast})$. Therefore, $K(Z_n) = K(Z_n^{\ast})$, which by~\cref{lemma:strict_tree_order} means that $f(Z_n) = f(Z_n^{\ast})$, so that $Z_n = Z_n^{\ast}$ because $f$ is a bijection.
\end{proof}

\medskip
With~\cref{thm:strict_max_equals_ast} in hand, we have obtained the maximal-rank tree with $(n-1)(k-1)+1$ leaves. For its rank, because $Z_n = Z_n^{\ast}$ and $f(Z_n^{\ast}) = 2 + \binom{f(Z_{n-1}^{\ast}) + k - 2}{k}$ for $n > 1$, we conclude that 
\begin{align} 
\label{eq:recurrence_bn}
b_n = 2 + \binom{b_{n-1}+k-2}{k},
\end{align}
with $b_1=1$. Note that setting $k=2$ yields Theorem 9 of~\cite{Rosenberg2021}. Structurally, $Z_n$ is a generalized version of the caterpillar tree, which also agrees with Corollary 10 of~\cite{Rosenberg2021}, stating that the caterpillar has maximal rank in the bifurcating case. For $k=3$, the trees $Z_n$ appear in Table~\ref{tab:Zn-zn-trees} for small values of $n$. The growth of $\{b_n\}$ is fast; for small $k$, the first several terms appear in Table~\ref{tab:bn}. 

\begin{table}[tb]
\centering
\begin{tabular}[H]{|c|c|c|T{2cm}|c|c|T{2cm}|} \hline
    & \multicolumn{3}{|c|}{Trees of maximal rank ($Z_n)$} & \multicolumn{3}{|c|}{Trees of minimal rank $(z_n)$} \\ \hline 
$n$ & $f(Z_n)$ & $K(Z_n)$ & $Z_n$ & $f(z_n)$ & $K(z_n)$ & $z_n$ \\ \hline
$1$ & $1$ & $(1,0,0)$ & \begin{forest}[[]]\end{forest} &  $1$ & $(1,0,0)$ & \begin{forest}[[]]\end{forest} \\
$2$ & $2$ & $(1,1,1)$ & \begin{forest}[[][][]]\end{forest} &  $2$ &  $(1,1,1)$ & \begin{forest}[[][][]]\end{forest}\\
$3$ & $3$ & $(2,1,1)$ & \begin{forest}[[[][][]][][]]\end{forest} & $3$ & $(2,1,1)$ & \begin{forest}[[[][][]][][]]\end{forest}\\
$4$ & $6$ & $(3,1,1)$ & \begin{forest}[[[[][][]][][]][][]]\end{forest} &  $4$ & $(2,2,1)$ & \begin{forest}[[[][][]][[][][]][]] \end{forest}\\
$5$ & $37$ & $(6,1,1)$ & \begin{forest}[[[[[][][]][][]][][]][][]]\end{forest} &  $5$ & $(2,2,2)$ & \begin{forest}[[[][][]][[][][]][[][][]]][]]\end{forest}\\
$6$ & $8438$ & $(37,1,1)$ & \begin{forest}[[[[[[][][]][][]][][]][][]][][]]\end{forest} & $8$ & $(3,2, 2)$ & \begin{forest} [[[[][][]][][]][[][][]][[][][]]]\end{forest}\\
$7$ & $100130712541$ & $(8438,1,1)$ & \begin{forest}[[[[[[[][][]][][]][][]][][]][][]][][]]\end{forest} & $10$ & $(3,3, 2)$ & \begin{forest} [[[[][][]][][]][[[][][]][][]][[][][]]]\end{forest} \\
\hline 
\end{tabular}
\caption{Maximal-rank and minimal-rank strictly trifurcating trees with $(n-1)(3-1)+1=2n-1$ leaves. For maximal-rank trifurcating tree $Z_n$, the rank $f(Z_n)$ and the list $K(Z_n)$ of ranks of the subtrees are shown. The trees $Z_n$ follow a generalized caterpillar shape. The minimal-rank trifurcating tree $z_n$ and its associated rank $f(z_n)$ and list $K(z_n)$ of ranks of its subtrees are also shown. The trees $z_n$ are highly balanced. We write $K(t)=(1,0,0)$ for the tree with 1 leaf.}
\label{tab:Zn-zn-trees}
\end{table}

\subsection{Minimal rank} 
\label{subsec:strictly_k_minimal_rank}

Next, we identify the strictly $k$-furcating trees with minimal rank among trees with $(n-1)(k-1)+1$ leaves. We employ a similar strategy: we construct trees $\{z_n^{\ast}\}_{n=1}^{\infty}$ and use the properties of $z_n^{\ast}$ and $z_n$ to prove that $z_n = z_n^{\ast}$. The construction proceeds as follows: 
(i) $z_1^{\ast}$ is the tree with one leaf. 
(ii) For $n > 1$, $z_{n}^{\ast}$ is the tree whose leftmost $n-2+k-k\lceil \frac{n-2}{k} \rceil$ subtrees are $z_{\lceil (n-2)/k \rceil + 1}^{\ast}$, and whose rightmost $-n+2 + k \lceil \frac{n-2}{k} \rceil$ subtrees are $z_{\lceil (n-2)/k \rceil}^{\ast}.$

We first verify that $z_n^{\ast}$ has $(n-1)(k-1)+1$ leaves, as it is not obvious from the definition. 
\begin{lemma}
$m(z_n^{\ast}) = (n-1)(k-1) + 1$ for $n \geq 1$.
\end{lemma}
\begin{proof}
We proceed by induction. In the base case, $m(z_1^{\ast})=(1-1)(k-1)+1 = 1$. For the inductive hypothesis, assume that $z_j^{\ast}$ has $(j-1)(k-1)+1$ leaves for all $j$, $1 \leq j < n$. Now count the leaves of $z_n^{\ast}$ using the recursive definition and the inductive hypothesis: 
\begin{align*}
m(z_n^{\ast}) &= \bigg(n-2+k - k\bigg\lceil \frac{n-2}{k} \bigg\rceil\bigg) m(z^{\ast}_{\lceil \frac{n-2}{k} \rceil + 1}) + \bigg(-n+2+ k\bigg\lceil \frac{n-2}{k} \bigg \rceil \bigg) m(z^{\ast}_{\lceil \frac{n-2}{k} \rceil}) \\ 
&= \bigg(n-2+k - k\bigg\lceil \frac{n-2}{k} \bigg\rceil\bigg) \bigg(\bigg\lceil \frac{n-2}{k} \bigg\rceil (k-1) + 1 \bigg) \\ 
& \quad + \bigg(-n+2+k \bigg\lceil \frac{n-2}{k} \bigg\rceil \bigg) \bigg[ \bigg(\bigg\lceil \frac{n-2}{k} \bigg\rceil - 1\bigg)(k-1) + 1 \bigg] \\ 
&= (n-2)(k-1) + k = (n-1)(k-1) + 1.
\end{align*}
The proof is complete. 
\end{proof}

\medskip
We next show that $f(z_n^{\ast})$ strictly increases with $n$.
\begin{proposition}
\label{prop:strict_min_increasing}
For $n \geq 2$, $f(z_n^{\ast}) > f(z_{n-1}^{\ast})$.
\end{proposition}
\begin{proof}
\sloppy
We induct on $n$. In the base case of $n=2$, 
$$K(z_2^{\ast}) = \Big(f(z_1^{\ast}), f(z_1^{\ast}), \overbrace{f(z_1^{\ast}), \ldots, f(z_1^{\ast})}^{k-2 \text{ times}}\Big) = (1,1,\overbrace{1, \ldots, 1}^{k-2 \text{ times}}).$$ By~\cref{eq:rank_strictly_k}, $f(z_2^{\ast}) = 2 > 1 = f(z_1^{\ast})$.

For the inductive hypothesis, suppose that $f(z_k^{\ast}) > f(z_{k-1}^{\ast})$ for all $k$ with $2 \leq k < n$. It then follows that 
\begin{align} 
\label{eq:Kzn}
K(z_n^{\ast}) &= \left(\overbrace{f(z_{\lceil \frac{n-2}{k} \rceil + 1}^{\ast}), \ldots, f(z_{\lceil \frac{n-2}{k} \rceil + 1}^{\ast})}^{n-2+k-k\lceil \frac{n-2}{k} \rceil \text{ times}}, \overbrace{f(z_{\lceil \frac{n-2}{k} \rceil }^{\ast}), \ldots, f(z_{\lceil \frac{n-2}{k} \rceil }^{\ast})}^{-n+2+k\lceil \frac{n-2}{k} \rceil \text{ times}} \right), \\
\label{eq:Kznminus1}
K(z_{n-1}^{\ast}) &= \left(\overbrace{f(z_{\lceil \frac{n-3}{k} \rceil + 1}^{\ast}), \ldots, f(z_{\lceil \frac{n-3}{k} \rceil + 1}^{\ast})}^{n-3+k-k\lceil \frac{n-3}{k} \rceil \text{ times}}, \overbrace{f(z_{\lceil \frac{n-3}{k} \rceil }^{\ast}), \ldots, f(z_{\lceil \frac{n-3}{k} \rceil }^{\ast})}^{-n+3+k\lceil \frac{n-3}{k} \rceil \text{ times}} \right).
\end{align}
We have two cases: (i) $\lceil \frac{n-2}{k} \rceil = \lceil \frac{n-3}{k} \rceil$; (ii) $\lceil \frac{n-2}{k} \rceil > \lceil \frac{n-3}{k} \rceil$. Observe that in case (ii), $k$ divides $n-3$, but in case (i), $k$ does not divide $n-3$.

\begin{table}[tb]
\centering
\small
\begin{tabular}{|c | c | c | c | c|}
\hline
    & \multicolumn{4}{c|}{$k$} \\ \cline{2-5}
$n$ &    2 &    3  &  4 &   5 \\ \hline
1   &    1 &    1  &  1 &   1 \\ 
2   &    2 &    2  &  2 &   2 \\ 
3   &    3 &    3  &  3 &   3 \\ 
4   &    5 &    6  &  7 &   8 \\ 
5   &   12 &   37 & 128 & 464 \\ 
6   &   68 & 8438 & 11358882 & 181164656830 \\ 
7   & 2280 & 100130712541 & 693635299649817827360747003 & 1626245591794207834538411826112599548105390018639492498  \\ 
\hline
\end{tabular}
\vspace{-.1cm}
\caption{The maximal rank $b_n$ of strictly $k$-furcating trees with $(n-1)(k-1)+1$ leaves, as obtained by~\cref{eq:recurrence_bn}. For $k=2$, the values of $b_n$ follow sequence A108225 in the On-Line Encyclopedia of Integer Sequences.}
\label{tab:bn}
\end{table}

\textit{Case (i): $\lceil \frac{n-2}{k} \rceil = \lceil \frac{n-3}{k} \rceil$}. In this case, we have 
\begin{align*}
K(z_n^{\ast}) &= \left(\overbrace{f(z_{\lceil \frac{n-2}{k} \rceil + 1}^{\ast}), \ldots, f(z_{\lceil \frac{n-2}{k} \rceil + 1}^{\ast})}^{n-2+k-k\lceil \frac{n-2}{k} \rceil \text{ times}}, \overbrace{f(z_{\lceil \frac{n-2}{k} \rceil }^{\ast}), \ldots, f(z_{\lceil \frac{n-2}{k} \rceil }^{\ast})}^{-n+2+k\lceil \frac{n-2}{k} \rceil \text{ times}} \right), \\
K(z_{n-1}^{\ast}) &= \left(\overbrace{f(z_{\lceil \frac{n-2}{k} \rceil + 1}^{\ast}), \ldots, f(z_{\lceil \frac{n-2}{k} \rceil + 1}^{\ast})}^{n-3+k-k\lceil \frac{n-2}{k} \rceil \text{ times}}, \overbrace{f(z_{\lceil \frac{n-2}{k} \rceil }^{\ast}), \ldots, f(z_{\lceil \frac{n-2}{k} \rceil }^{\ast})}^{-n+3+k\lceil \frac{n-2}{k} \rceil \text{ times}} \right).
\end{align*}
Because $n-3+k-k\lceil \frac{n-2}{k} \rceil=n-3+k-k\lceil \frac{n-3}{k} \rceil > n-3+k-k[1 + \frac{n-3}{k}] = 0$, at least one $f(z^{\ast}_{\lceil (n-2)/k \rceil+1})$ term appears in both $K(z_n^{\ast})$ and $K(z_{n-1}^{\ast})$. Compared to $K(z_{n-1}^*)$, $K(z_n^*)$ contains one additional term $f(z_{\lceil (n-2)/k \rceil + 1}^{\ast})$; $K(z_{n-1}^*)$ has one additional term $f(z_{\lceil (n-2)/k \rceil}^{\ast})$ relative to $K(z_n^*)$. Because $\lceil \frac{n-2}{k} \rceil + 1< n$ for $n \geq 2$ and $k \geq 2$, the inductive hypothesis implies that $f(z_{\lceil (n-2)/k \rceil+1}^{\ast}) > f(z_{\lceil (n-2)/k \rceil}^{\ast})$. The additional term $f(z_{\lceil (n-2)/k \rceil + 1}^{\ast})$ for $K(z_n^*)$ has the consequence that $K(z_{n}^{\ast})$ is lexicographically greater than $K(z_{n-1}^{\ast})$, $K(z_{n}^{\ast}) >_D K(z_{n-1}^{\ast})$. By~\cref{lemma:strict_tree_order}, we conclude $f(z_n^{\ast}) > f(z_{n-1}^{\ast})$.

\textit{Case (ii): $\lceil \frac{n-2}{k} \rceil > \lceil \frac{n-3}{k} \rceil$}. As in case (i), we note that $\lceil \frac{n-2}{k} \rceil + 1< n$ for $n \geq 2$ and $k \geq 2$. Therefore, by the inductive hypothesis, $f(z_{\lceil (n-2)/k \rceil+1}^{\ast}) > f(z_{\lceil (n-3)/k \rceil + 1}^{\ast})$. We immediately see that $K(z_n^*)$ (\cref{eq:Kzn}) is lexicographically greater than $K(z_{n-1}^*)$ (\cref{eq:Kznminus1}), $K(z_n^*) >_D K(z_{n-1}^*)$, which implies that $f(z_n^{\ast}) > f(z_{n-1}^{\ast})$ by~\cref{lemma:strict_tree_order}. 
\end{proof}

\medskip 

Equipped with~\cref{prop:strict_min_increasing}, we now prove that $z_n = z_n^{\ast}$. In other words, the $k$-furcating tree with minimal rank on $(n-1)(k-1)+1$ leaves is exactly $z_n^{\ast}$. We defer the proof to Appendix~\ref{pf:strictly_min_equals_ast}.
\begin{theorem} 
For $n\geq 1$, the strictly $k$-furcating tree $z_n$ with minimal rank is $z_n = z_n^{\ast}$.
\label{thm:strictly_min_equals_ast} 
\end{theorem} 

We complete this section by explicitly computing $f(z_n^{\ast})$. Beginning from \cref{eq:Kzn}, by~\cref{eq:rank_strictly_k} and the ``hockey-stick identity,'' $\sum_{i=0}^{n-r} \binom{i+r}{i} = \binom{n+1}{n-r}$ for $n \geq r$~\cite[eq.~5.9]{GrahamEtAl94},
\begin{align*}
f(z_n^{\ast}) &= 2 + \sum_{i=1}^{-n+2+k \lceil \frac{n-2}{k} \rceil} \binom{f(z_{\lceil \frac{n-2}{k} \rceil}^{\ast})+i-2}{i}+\sum_{i=-n+3+k \lceil \frac{n-2}{k} \rceil }^{k} \binom{f(z_{\lceil \frac{n-2}{k} \rceil+1}^{\ast})+i-2}{i} \\ 
&= 2 + \bigg[\binom{f(z_{\lceil \frac{n-2}{k} \rceil}^{\ast}) + 1 - n + k \lceil \frac{n-2}{k} \rceil}{-n+2+ k \lceil \frac{n-2}{k} \rceil} - 1\bigg] \\ 
& \qquad + \sum_{i=1}^{k} \binom{f(z_{\lceil \frac{n-2}{k} \rceil+1}^{\ast})+i-2}{i} - \sum_{i=1}^{-n+2+k \lceil \frac{n-2}{k} \rceil} \binom{f(z_{\lceil \frac{n-2}{k} \rceil+1}^{\ast})+i-2}{i} \\
&= 1 + \binom{f(z_{\lceil \frac{n-2}{k} \rceil}^{\ast}) + 1 - n + k \lceil \frac{n-2}{k} \rceil}{-n+2+ k \lceil \frac{n-2}{k} \rceil} + \binom{f(z_{\lceil \frac{n-2}{k} \rceil + 1}^{\ast}) +k-1}{k} \\ 
& \qquad - \binom{f(z_{\lceil \frac{n-2}{k} \rceil + 1}^{\ast}) + 1 - n + k \lceil \frac{n-2}{k} \rceil}{-n+2+ k \lceil \frac{n-2}{k} \rceil}.
\end{align*}
Because $a_n = f(z_n) = f(z_n^{\ast})$, we have the recurrence
\begin{align} 
\label{eq:recurrence_an}
a_n = 1 + \binom{a_{\lceil \frac{n-2}{k} \rceil} + 1 - n + k \lceil \frac{n-2}{k} \rceil}{-n+2+ k \lceil \frac{n-2}{k} \rceil} + \binom{a_{\lceil \frac{n-2}{k} \rceil + 1} +k-1}{k} - \binom{a_{\lceil \frac{n-2}{k} \rceil + 1} + 1 - n + k \lceil \frac{n-2}{k} \rceil}{-n+2+ k \lceil \frac{n-2}{k} \rceil},
\end{align}
with $a_1=1$ and $a_2=2$. This recurrence yields~\cite[Theorem 6]{Rosenberg2021} for $k=2$. Note that if $k$ evenly divides $n - 2$, $k | (n-2)$, then $-n+2+k \lceil \frac{n-2}{k} \rceil = 0,$ so that the first and third binomial terms cancel, and 
\begin{equation}
a_n = 1 + \binom{a_{\frac{n-2}{k} + 1} + k - 1}{k}.
\label{eq:cancel}
\end{equation}

For $k=3$, the trees $z_n$ appear in Table~\ref{tab:Zn-zn-trees} for small values of $n$. The first several terms of $\{a_n\}$ for small values of $k$ appear in~\cref{tab:an}. The growth of $\{a_n\}$ appears to be far slower than that of $\{b_n\}$. 

\begin{table}[tb]
\centering
\small
\begin{tabular}{|c | c | c | c | c|}
\hline
    & \multicolumn{4}{c|}{$k$} \\ \cline{2-5}
$n$ &   2   &     3 &    4 &     5 \\ 
\hline
1   &   1   &     1 &    1 &     1 \\
2   &   2   &     2 &    2 &     2 \\
3   &   3   &     3 &    3 &     3 \\
4   &   4   &     4 &    4 &     4 \\
5   &   6   &     5 &    5 &     5 \\
6   &   7   &     8 &    6 &     6 \\
7   &   10  &    10 &   10 &     7 \\
8   &   11  &    11 &   13 &    12 \\
9   &   20  &    17 &   15 &    16 \\
10  &   22  &    20 &   16 &    19 \\
11  &   28  &    21 &   26 &    21 \\
12  &   29  &    31 &   32 &    22 \\
13  &   53  &    35 &   35 &    37 \\
14  &   56  &    36 &   36 &    47 \\
15  &   66  &   100 &   56 &    53 \\
16  &   67  &   118 &   66 &    56 \\
17  &   202 &   121 &   70 &    57 \\
18  &   211 &   202 &   71 &    92 \\
19  &   252 &   219 &  106 &   112 \\
20  &   254 &   221 &  121 &   122 \\
\hline
\end{tabular}
\vspace{-.2cm}
\caption{The minimal rank $a_n$ of strictly $k$-furcating trees with $(n-1)(k-1)+1$ leaves, as obtained by~\cref{eq:recurrence_an}. For $k=2$, the values of $a_n$ follow sequence A354970 in the On-Line Encyclopedia of Integer Sequences.}
\label{tab:an}
\end{table}

\subsection{Asymptotics} 
\label{subsec:strictly_k_asymptotics}

We next study the asymptotic growth of $a_n$ and $b_n$ given their recurrences provided by~\cref{eq:recurrence_an} and ~\cref{eq:recurrence_bn}. We begin with a relation between the minimal and maximal rank.

\begin{proposition} 
\label{prop:minimal_rank_maximal_rank_related} $a_{1 + \frac{k^n - 1}{k-1}} + 1 = b_{n+2}$ for $n \geq 0$.
\end{proposition} 
\begin{proof}
We first see that $1 + (k^n-1)/(k-1) = 2 + (k + \ldots + k^{n-1}),$ so that $k | \big(1 + (k^n-1) / (k-1) - 2\big)$. By \cref{eq:cancel}, the recursive formula for $a_{1 + (k^n - 1)/(k-1)}$ is 
\begin{align} 
\label{eq:series_a_recursive}
a_{1 + \frac{k^n - 1}{k-1}} = 1 + \binom{a_{\frac{k+k^2+\ldots+k^{n-1}}{k} + 1} + k - 1}{k} = 1 + \binom{a_{\frac{k^{n-1} - 1}{k - 1} + 1} + k - 1}{k}.
\end{align}

We induct on $n$. For the base case of $n=0$, the statement is that $a_1+1=b_2$, which holds as $a_1=1$ and $b_2=2$. For $n=1$, the statement is that $a_2 + 1 = b_3$; the equality holds, as $a_2 = 2$ and $b_3 = 3$. 

For the inductive hypothesis, for each $n > 1$, assume that $a_{1 + (k^{n-1} - 1)/(k - 1)} + 1 = b_{n+1}$. Using~\cref{eq:recurrence_bn} and~\cref{eq:series_a_recursive},
\begin{align*}
& a_{1 + \frac{k^n - 1}{k-1}} + 1 = 1 + \bigg[1 + \binom{a_{\frac{k^{n-1} - 1}{k - 1} + 1} + k - 1}{k} \bigg] = 2 + \binom{b_{n+1} - 1 + k - 1}{k} = 2 + \binom{b_{n+1} + k - 2}{k} = b_{n+2}.
\end{align*}
The proof is complete. 
\end{proof}

\medskip
Examples of \cref{prop:minimal_rank_maximal_rank_related} can be seen in~\cref{tab:bn} and~\cref{tab:an}. In the case of $k=2$, $a_8+1=b_5=12$ and $a_{16}+1=b_6=68$. For $k=3$, $a_5+1 = b_4=6$ and $a_{14}+1=b_5=37$. A corollary of~\cref{prop:minimal_rank_maximal_rank_related} is that the trees with rank in the interval $[b_{n}, b_{n+1})$ have number of leaves in $[n, 1 + (k^{n-1}-1)/(k-1)]$. This result can be obtained by noting first that $a_{1 + (k^{n-1}-1)/(k-1)}=b_{n+1}-1$. Because $\{a_n\}_{n=1}^{\infty}$ is increasing by~\cref{prop:strict_min_increasing} and~\cref{thm:strictly_min_equals_ast}, each tree with rank in $[b_n, b_{n+1})$ has at most $1 + (k^{n-1}-1)/(k-1)$ leaves. Each tree with rank in $[b_n, b_{n+1})$ has at least $n$ leaves because $\{b_n\}_{n=1}^{\infty}$ is increasing by~\cref{prop:strict_monotonic} and~\cref{thm:strict_max_equals_ast}. This result generalizes Proposition 12 of~\cite{Rosenberg2021} to general $k \geq 2$.

We next study the asymptotic growth of $b_n$, which we show grows as a doubly exponential function. The result explains the fast growth exhibited for small $n$ (Table~\ref{table: beta-alpha}). The base of the exponential depends on $k$.

\begin{theorem} 
\label{thm:asymptotic_b_n}
$b_n \sim (k!)^{\frac{1}{k-1}} \beta_k^{(k^n)}$ for a constant $\beta_k$ depending only on $k \geq 2$.
\end{theorem}
\begin{proof} 
Rewrite the recurrence $b_n = 2 + \binom{b_{n-1}+k-2}{k}$ in \cref{eq:recurrence_bn} as $b_n - 2 = \binom{(b_{n-1}-2) + k}{k}$, or $d_n = \binom{d_{n-1} + k}{k}$, where $d_n = b_n - 2$ for all $n \geq 3$. Because $b_3 = 3$, we have that $d_3 = 1$. Clearly $d_n \sim b_n$, so it suffices to show that $d_n \sim (k!)^{\frac{1}{k-1}} \beta_k^{(k^n)}$.

We simplify the recurrence further: 
\begin{align*}
d_n &= \binom{d_{n-1} + k}{k} = \frac{\big(d_{n-1}+k\big)\big(d_{n-1}+k-1\big) \cdots \big(d_{n-1}+1\big)}{k!} \\
&= \frac{1}{k!}d_{n-1}^{k} + \frac{1}{k!}P_k(d_{n-1}),
\end{align*}
where $P_k(x) = \big(x+k \big)\big(x+k-1\big) \cdots \big(x+1\big) - x^k$ is a polynomial with positive integer coefficients and degree $k-1$. For example, if $k = 4$, then $P_4(x) = (x+4)(x+3)(x+2)(x+1)-x^4 = 10x^3 + 35x^2 + 50x + 24$. 

Dividing both sides of the recurrence by $d_{n-1}^k$, we deduce that $k! \,{d_n}/{d_{n-1}^k} = 1 + {P_k(d_{n-1})}/{d_{n-1}^k}$. Taking the logarithm of both sides of this equation and letting $y_n = \log d_n$, we find that for $n\geq 3$,
$$y_n = ky_{n-1} - \log(k!) + \rho_{n-1}, $$
with $\rho_{n-1} = \log \big[1 + P_k(d_{n-1}) / d_{n-1}^k \big]$. 

Note that $\rho_{n-1} \to 0$ as $n \to \infty$, because the polynomial $P_k$ has degree less than $k$ and $d_n$ grows arbitrarily large. We can eliminate the $\log(k!)$ term via the substitution $z_n = y_n - \log(k!) / (k-1)$, yielding for $n\geq 3$,
\begin{align*}
z_n = kz_{n-1} + \rho_{n-1}.
\end{align*}

Now, following the method of Aho \& Sloane~\cite{AhoSloane}, for $n \geq 4$,
\begin{align} 
\label{eq:aho_sloane_recurrence_strictly_k}
& z_n = k^{n-3}z_3 + \sum_{i=3}^{n-1} k^{n - 1 - i}\rho_i = k^n \bigg[z_3k^{-3} + \sum_{i=3}^{\infty} k^{-(i+1)} \rho_i \bigg] - \sum_{i=n}^{\infty} k^{n-1-i} \rho_i. 
\end{align}
Exponentiating both sides,
\begin{align*}
d_n \exp\bigg[-\frac{\log(k!)}{k-1} \bigg] & = e^{z_n} = \bigg[\exp(z_3k^{-3}) \, \exp \bigg(\sum_{i=3}^{\infty} k^{-(i+1)}\rho_i \bigg) \bigg]^{(k^n)} \exp \bigg( - \sum_{i=n}^{\infty} k^{n-1-i} \rho_i \bigg) \\ 
&= \beta_k^{(k^{n})} \exp \bigg( - \sum_{i=n}^{\infty} k^{n-1-i} \rho_i \bigg),
\end{align*}
where $\beta_k$ is the constant defined by $\beta_k = \exp(z_3k^{-3}) \, \exp (\sum_{i=3}^{\infty} k^{-(i+1)}\rho_i )$. Hence, 
\begin{align*} 
\label{eq:ratio_d_n_beta}
\frac{d_n}{\beta_k^{(k^n)}} = \exp \bigg[\frac{ \log(k!)}{k-1} \bigg] \exp \bigg( - \sum_{i=n}^{\infty} k^{n-1-i} \rho_i \bigg) = (k!)^{\frac{1}{k-1}} \exp \bigg( - \sum_{i=n}^{\infty} k^{n-1-i} \rho_i \bigg).
\end{align*}

Next, note that $\rho_n$ strictly decreases as $n$ increases. To verify this claim, recall that $\rho_n = \log [1 + P_k(d_n) / d_n^k ]$. Writing $P_k(x) = a_{k-1}x^{k-1} + a_{k-2}x^{k-2} + \ldots + a_1x + a_0$ for positive constants $a_0,a_1,\ldots,a_{k-1}$,
\begin{align*}
\frac{P_k(d_n)}{d_n^k} 
= \frac{a_{k-1}}{d_n} + \ldots + \frac{a_1}{d_n^{k-1}} + \frac{a_0}{d_n^k}.
\end{align*} 
Because $d_n$ strictly increases without bound as $n$ increases and $a_0,a_1,\ldots,a_{k-1}$ are positive constants, $P_k(d_n) / d_n^k$ strictly decreases to zero as $n$ increases, so that $\rho_n$ is strictly decreasing as well.

Because $\rho_n \to 0$, we have $\sum_{i=n}^{\infty} k^{n-1-i} \rho_i \leq \rho_n \sum_{i=n}^{\infty} k^{n-1-i} = \rho_n / (k-1) \to 0$. We conclude ${d_n}/{\beta_k^{(k^n)}} \sim (k!)^\frac{1}{k-1}$, and the result follows.
\end{proof}

\medskip
With $b_n \sim d_n \sim (k!)^{\frac{1}{k-1}} \beta_k^{(k^n)}$, where $\beta_k = \exp(z_3k^{-3}) \, \exp (\sum_{i=3}^{\infty} k^{-(i+1)}\rho_i )$, we compute the constants $\beta_k$ for small $k$ from the first 12 terms of the sequence $\{b_n\}$. These constants appear in Table~\ref{table: beta-alpha}. For $k=2$, the numerical value $\beta_2$ agrees with~\cite{Rosenberg2021}. The decrease with $k$ of the values of $\beta_k$ in Table~\ref{table: beta-alpha} suggests a result that is proven in Appendix~\ref{sec:proof_beta_decreasing}.

\begin{table}[tb]
\begin{center}
\begin{tabular}{ |c|c|c| } 
\hline
$k$ & $\beta_k$ & $\alpha_k$ \\
\hline
$2$ & $1.056528765669$ & $1.246020832983$ \\
$3$ & $1.011234962848$ & $1.105779896530$ \\
$4$ & $1.003714439923$ & $1.061115735103$ \\
$5$ & $1.001583703886$ & $1.040354248550$ \\
$6$ & $1.000789140695$ & $1.028804924924$ \\
$7$ & $1.000437369864$ & $1.021657632449$ \\
\hline
\end{tabular}
\vspace{-.5cm}
\captionof{table}{Values of $\beta_k$ and $\alpha_k$, calculated using $\beta_k = \exp(z_3k^{-3}) \, \exp (\sum_{i=3}^{\infty} k^{-(i+1)}\rho_i )$ and $\alpha_k = \beta_k^{(k^2)}$. To calculate $\rho_i$ and $z_3$, we evaluate the first 12 terms $\{b_n\}$ via the recurrence~\cref{eq:recurrence_bn}. The value $k$ denotes the number of immediate descendants of each internal node; $\beta_k$ is  the base of the exponent in the growth $b_n \sim (k!)^{\frac{1}{k-1}} \beta_k^{(k^n)}$ and $\alpha_k$ is the base of the exponent in the growth $a_{1 + (k^n-1)/(k-1)} \sim (k!)^{\frac{1}{k-1}} \alpha_k^{(k^n)}$. Note that the calculation is accurate beyond 12 decimal places; we show the first 12 decimal places.}
\label{table: beta-alpha}
\end{center}
\end{table}

\begin{proposition} 
\label{lemma:beta_decreasing}
For $k\geq 2$, the constant $\beta_k$ satisfies $\beta_{k+1} < \beta_k$.
\end{proposition}

\cref{thm:asymptotic_b_n} has the following immediate corollary.
\begin{corollary}
For $k \geq 2$, $a_{1 + \frac{k^n-1}{k-1}} \sim (k!)^{\frac{1}{k-1}} \alpha_k^{(k^n)}$, where $\alpha_k = \beta_k^{(k^2)}$.
\end{corollary}
\begin{proof}
Because $a_{1 + \frac{k^n - 1}{k-1}} + 1 = b_{n+2}$ 
(\cref{prop:minimal_rank_maximal_rank_related}), we have
\begin{align*}
a_{1 + \frac{k^n - 1}{k-1}} \sim a_{1 + \frac{k^n - 1}{k-1}} + 1 = b_{n+2} \sim (k!)^{\frac{1}{k - 1}} \beta_k^{(k^{n+2})} = (k!)^{\frac{1}{k - 1}} \alpha_k^{(k^n)}.
\end{align*}
The result then follows. 
\end{proof}

\medskip
Values of $\alpha_k$ for small $k$ appear in Table~\ref{table: beta-alpha}. 

\section{At-most-$k$-furcating trees} 
\label{sec:at_most_k_furcating}

In this section, we perform on at-most-$k$-furcating trees a similar analysis to that we conducted with strictly $k$-furcating trees. One difference between the at-most-$k$-furcating and strictly $k$-furcating cases is that in contrast with the strictly $k$-furcating case, where we require that $n \equiv 1 \pmod {k-1}$, for each $n \geq 1$, an at-most-$k$-furcating tree with $n$ leaves can be constructed. For example, the caterpillar tree with $n$ leaves is an at-most-$k$-furcating tree for $k \geq 2$, whereas it is not a strictly $k$-furcating tree for $k > 2$.

Denote by $A_n$ and $B_n$ the minimal and maximal rank among at-most-$k$-furcating trees with $n$ leaves. Let $y_n$ and $Y_n$ be the at-most-$k$-furcating trees with ranks $A_n$ and $B_n$, respectively. 

\subsection{Maximal rank} 
\label{subsec:at_most_k_maximal_rank}

The strategy for finding the trees $Y_n$ is similar to that used in~\cref{sec:strictly_k}. We construct a collection of trees $\{Y_n^{\ast}\}$ and show via induction that $Y_n = Y_n^{\ast}$. Set $\{Y_n^{\ast}\}_{n \geq 1}$ to be the collection of bifurcating caterpillar trees: (i) $Y_1^{\ast}$ is the tree with exactly one leaf; (ii) $Y_n^{\ast}$ is the tree with subtrees $Y_{n-1}^{\ast}$ and $Y_1^{\ast}$. We show that the rank $f(Y_n^{\ast})$ is strictly increasing with $n$.

\begin{proposition}
\label{prop:up_to_k_maximal_increasing} 
$f(Y_n^{\ast}) > f(Y_{n-1}^{\ast})$ for all $n \geq 2$.
\end{proposition}
\begin{proof}
We proceed via induction. For $n=1$, the caterpillar tree $Y_1^{\ast}$ has $f(Y_1^{\ast})=1$. In the base case $n=2$, 
$$K(Y_2^{\ast}) = \Big(f(Y_1^{\ast}), f(Y_1^{\ast}), \overbrace{0, 0, \ldots, 0}^{k-2 \text{ times }}\Big) = (1,1, \overbrace{0, 0, \ldots, 0}^{k-2 \text{ times }}).$$ 
By~\cref{eq:rank_at_most_k},
\begin{align*}
f(Y_2^{\ast}) = -1 + 1 + \binom{1 + (k-1) - 1}{k-1} + \binom{1 + k - 1}{k} = 2.
\end{align*}
Then $f(Y_2^{\ast}) = 2 > 1 = f(Y_1^{\ast})$.

For the inductive hypothesis, assume that for some positive integer $n$, $f(Y_\ell^{\ast}) > f(Y_{\ell-1}^{\ast})$ for all $\ell$, $2 \leq \ell \leq n-1$. By our recursive definition of $Y_n^{\ast}$, $K(Y_n^{\ast}) = \big(f(Y_{n-1}^{\ast}), f(Y_1^{\ast}), 0, 0, \ldots, 0 \big)$ and $K(Y_{n-1}^{\ast}) = \big(f(Y_{n-2}^{\ast}), f(Y_1^{\ast}), 0, 0, \ldots, 0 \big)$. By the inductive hypothesis, $f(Y_{n-1}^{\ast}) > f(Y_{n-2}^{\ast})$, so that $K(Y_{n}^{\ast}) >_D K(Y_{n-1}^{\ast})$. Hence, $f(Y_{n}^{\ast}) > f(Y_{n-1}^{\ast})$ by~\cref{lemma:up_to_k_tree_order}. 
\end{proof}

\medskip
We use~\cref{prop:up_to_k_maximal_increasing} to prove that the at-most-$k$-furcating tree of maximal rank on $n$ leaves is $Y_n=Y_n^{\ast}$, demonstrating that the at-most-$k$-furcating tree attaining the maximal rank on $n$ leaves is the (bifurcating) caterpillar tree. 

\begin{theorem}
For $n\geq 1$, the at-most-$k$-furcating tree $Y_n$ with maximal rank is $Y_n = Y_n^{\ast}$. 
\label{thm:at_most_k_equals_ast}
\end{theorem}
\begin{proof}
We proceed via induction. The base case of $n=1$ is trivial: only one tree has one leaf, and $Y_1 = Y_1^{\ast}$. 

For the inductive hypothesis, assume that for some positive integer $n$, $Y_\ell = Y_\ell^{\ast}$ for all $\ell$, $1 \leq \ell \leq n-1$. Suppose that $Y_n$ has subtrees $a_1,a_2, \ldots, a_k$ in canonical ordering. Note that because $Y_k$ is an at-most-$k$-furcating tree, some of $a_3, a_4, \ldots, a_k$ might be empty. We have that $K(Y_n) = \big( f(a_1), f(a_2), \ldots, f(a_k) \big)$. Because $Y_n$ is the tree of maximal rank on $n$ leaves, we must have that $f(Y_n) \geq f(Y_n^{\ast})$, so that $K(Y_n) \geq_D K(Y_n^{\ast})$. In particular, $\big( f(a_1), f(a_2), \ldots, f(a_k) \big) \geq_D \big( f(Y_{n-1}^{\ast}), 1, 0, \ldots, 0) \big)$, and $f(a_1) \geq f(Y_{n-1}^{\ast})$.

\clearpage
\begin{table}[tb]
\centering
\begin{tabular}[H]{|c|c|c|T{2cm}|c|c|T{2cm}|} \hline
\multicolumn{1}{|c|}{} & \multicolumn{3}{|c|}{Trees of maximal rank ($Y_n)$} & \multicolumn{3}{|c|}{Trees of minimal rank $(y_n)$} \\
\hline 
$n$ & $f(Y_n)$& $K(Y_n) $& $Y_n$ & $f(y_n)$ & $K(y_n)$ & $y_n$ \\ \hline
1 & 1 & $(1,0,0)$ & \begin{forest}[[]]\end{forest} & 1 & $(1,0,0)$ & \begin{forest}[[]]\end{forest}\\
2 & 2 & $(1,1,0)$ & \begin{forest}[[][]]\end{forest} & 2 & $(1,1,0)$ & \begin{forest}[[][]]\end{forest}\\
3 & 4 & $(2,1,0)$ & \begin{forest}[[[][]][]]\end{forest} & 3 & $(1,1,1)$ & \begin{forest}[[][][]]\end{forest}\\
4 & 18 & $(4,1,0)$ & \begin{forest}[[[[][]][]][]]\end{forest} & 5 & $(2,1,1)$ & \begin{forest}[[[][]][][]] \end{forest}
\\
5 & 1124 & $(18,1,0)$ & \begin{forest}[[[[[][]][]][]][]]\end{forest} & 7 & $(2,2,1)$ & \begin{forest}[[[][]][[][]][]]\end{forest}\\
6 & 237303378 & $(1124,1,0)$ & \begin{forest}[[[[[[][]][]][]][]][]]\end{forest} & 8 & $(2,2, 2)$ & \begin{forest} [[[][]][[][]][[][]]]\end{forest}\\
7 & $2.227 \times 10^{24}$ & $(237303378,1,0)$ & \begin{forest}[[[[[[[][]][]][]][]][]][]]\end{forest} & 13 & $(3,2, 2)$ & \begin{forest} [[[][][]][[][]][[][]]]\end{forest} \\
\hline 
\end{tabular}
\vskip -.2cm
\caption{Maximal-rank and minimal-rank at-most-trifurcating trees with $n$ leaves. For maximal-rank at-most-trifurcating tree $Y_n$, the rank $f(Y_n)$ and the list $K(Y_n)$ of ranks of the subtrees are shown. The trees $Y_n$ are bifurcating caterpillars. The minimal-rank at-most-trifurcating tree $y_n$ and its associated rank $f(y_n)$ and list $K(y_n)$ of ranks of its subtrees are also shown. The trees $y_n$ are highly balanced. The value of $f(Y_7)$ is an approximation.}
\label{tab:Yn-yn-trees}
\end{table}

We also have that 
\begin{align*}
f(a_1) \leq f(Y_{m(a_1)}^{\ast}) \leq f(Y_{n-1}^{\ast}),
\end{align*}
where the first inequality follows from the inductive hypothesis applied to $a_1$, as $a_1$ has strictly fewer than $n$ leaves, and the second inequality follows from~\cref{prop:up_to_k_maximal_increasing}. Therefore, $f(a_1) = f(Y_{n-1}^{\ast})$, so that $a_1$ and $Y_{n-1}^{\ast}$ are the same tree. Because $a_1$ has $n-1$ leaves and $Y_n$ has a total of $n$ leaves, it follows that $a_2$ must have one leaf and $a_3, a_4, \ldots, a_k$ are empty. Hence, $Y_n$ is the tree that has $Y_{n-1}^{\ast}$ and $Y_1^{\ast}$ as its two non-empty subtrees, which means that $Y_n = Y_n^{\ast}$.
\end{proof}
\medskip
\cref{thm:at_most_k_equals_ast} characterizes the maximal rank among at-most-$k$-furcating trees on $n$ leaves. Because $Y_n = Y_n^{\ast}$ and $f(Y_n^{\ast}) = 2-f(Y_{n-1}^{\ast}) + \binom{f(Y_{n-1}^{\ast})+k-1}{k}$, we conclude that 
\begin{align}
\label{eq:recurrence_Bn}
B_n &= 2 - B_{n-1} + \binom{B_{n-1} + k -1}{k},
\end{align}
with $B_1=1$. For $k=2$, \cref{eq:recurrence_Bn} is identical to \cref{eq:recurrence_bn}, as an at-most-2-furcating tree is equivalent to a strictly 2-furcating tree. For $k=3$, the trees $Y_n$ appear in Table~\ref{tab:Yn-yn-trees} for small values of $n$. Values of the rapidly growing $B_n$ for small $k$ appear in~\cref{tab:Bn}.

\begin{table}[tb]
\centering
\small
\begin{tabular}{|c | c | c | c | c|}
\hline
    & \multicolumn{4}{c|}{$k$} \\ \cline{2-5}
$n$ &    2 &  3  &  4 &   5 \\ 
\hline
1   &    1 &  1  &  1 &   1 \\ 
2   &    2 &  2  &  2 &   2 \\ 
3   &    3 &  4  &  5 &   6 \\ 
4   &    5 &  18 & 67 & 248 \\ 
5   &   12 &  1124 & 916830 & 8137369554 \\ 
6   &   68 &  237303378 & 29440613974007230765292 & 297330152006749281113411833884920659485794167214 \\ 
\hline
\end{tabular}
\vspace{-.2cm}
\caption{The maximal rank $B_n$ of at-most-$k$-furcating trees with $n$ leaves, as obtained by~\cref{eq:recurrence_Bn}. For $k=2$, the values of $B_n$ follow sequence A108225 in the On-Line Encyclopedia of Integer Sequences.}
\label{tab:Bn}
\end{table}

\subsection{Minimal rank} 
\label{subsec:at_most_k_minimal_rank}

We find the at-most-$k$-furcating tree that attains the minimal rank among trees with $n$ leaves. As in the strictly $k$-furcating case (\cref{subsec:strictly_k_minimal_rank}), the minimal-rank at-most-$k$-furcating tree is in a sense the most balanced tree possible. The strategy is to construct $\{y_n^{\ast}\}$ and to show that $y_n = y_n^{\ast}$ for each $n \geq 1$. 

Define $\{y_n^{\ast}\}$ recursively: (i) $y_1^{\ast}$ is the tree with one leaf; (ii) for $n > 1$, $y_n^{\ast}$ is the tree whose leftmost $n-k\lfloor n/k \rfloor$ subtrees are $y^{\ast}_{\lceil n/k \rceil}$ and whose rightmost $k-n +k \lfloor n/k \rfloor$ subtrees are $y^{\ast}_{\lfloor n/k \rfloor}$. As in the analysis of $z_n^{\ast}$ in~\cref{subsec:strictly_k_minimal_rank}, it is not clear from the recursive definition that $y_n^{\ast}$ has $n$ leaves. We verify this claim.

\begin{lemma}
$m(y_n^{\ast})=n$ for $n \geq 1$.
\end{lemma}
\begin{proof}
We induct on $n$. For the base case of $n=1$, $y_1^{\ast}$ has one leaf. For the inductive hypothesis, suppose that $y_{\ell}^{\ast}$ has $\ell$ leaves for all $\ell$, $1 \leq \ell \leq n-1$. By the recursive definition of $y_n^{\ast}$, 
\begin{align*}
m(y_n^{\ast}) & = \bigg(n - k \bigg \lfloor \frac{n}{k} \bigg \rfloor \bigg) m(y^{\ast}_{\lceil \frac{n}{k} \rceil}) + \bigg(k-n+k \bigg\lfloor \frac{n}{k} \bigg\rfloor\bigg) m(y^{\ast}_{\lfloor \frac{n}{k} \rfloor}) \\ 
&= \bigg(n - k \bigg\lfloor \frac{n}{k} \bigg\rfloor\bigg) \bigg\lceil \frac{n}{k} \bigg\rceil + \bigg(k-n+k \bigg\lfloor \frac{n}{k} \bigg\rfloor\bigg)\bigg\lfloor \frac{n}{k} \bigg\rfloor \\ 
&= n \bigg\lceil\frac{n}{k} \bigg\rceil - k \bigg\lfloor \frac{n}{k} \bigg\rfloor \bigg\lceil \frac{n}{k} \bigg\rceil + k \bigg\lfloor \frac{n}{k} \bigg\rfloor - n \bigg\lfloor \frac{n}{k} \bigg\rfloor + k \bigg\lfloor \frac{n}{k} \bigg\rfloor^2.
\end{align*}
If $k | n$, then this expression becomes 
$n ( \frac{n}{k} ) - k ( \frac{n}{k} )^2 + k ( \frac{n}{k} ) - n ( \frac{n}{k} ) + k ( \frac{n}{k} ) ( \frac{n}{k} ) = n$.
If $k \nmid n$, then $\lceil\frac{n}{k} \rceil = \lfloor \frac{n}{k} \rfloor + 1$, and instead it is 
$n ( \lfloor \frac{n}{k} \rfloor + 1 ) - k \lfloor \frac{n}{k} \rfloor ( \lfloor \frac{n}{k} \rfloor + 1 ) + k \lfloor \frac{n}{k} \rfloor - n \lfloor \frac{n}{k} \rfloor + k \lfloor \frac{n}{k} \rfloor^2 = n$.
\end{proof}

\medskip
We next show that $\{f(y^{\ast}_n)\}_{n=1}^{\infty}$ increases with $n$.
\begin{proposition} 
\label{prop:up_to_k_minimal_increasing}
For $n \geq 2$, $f(y^{\ast}_n) > f(y^{\ast}_{n-1})$.
\end{proposition}
\begin{proof}
We induct on $n$. In the base case of $n=2$, $y_2^{\ast}$ is the tree that has two copies of $y_1^{\ast}$ as its non-empty subtrees. In particular, $K(y_2^{\ast}) = \big(f(y_1^{\ast}), f(y_1^{\ast}), 0,\ldots,0\big) = (1,1,0,\ldots,0)$. Using~\cref{eq:rank_at_most_k}, we conclude that $f(y_2^{\ast}) = 2$, so that $f(y_2^{\ast}) = 2 > 1 = f(y_1^{\ast}).$

For the inductive hypothesis, suppose $f(y^{\ast}_\ell) > f(y^{\ast}_{\ell-1})$ for $2 \leq \ell \leq n-1$, with $n \geq 2$. It follows that 
\begin{align*}
K(y_n^{\ast}) &= \Big(\overbrace{f(y^{\ast}_{\lceil n/k \rceil}), \ldots,f(y^{\ast}_{\lceil n/k \rceil})}^{n - k \lfloor n/k \rfloor \text{ times }}, \overbrace{f(y^{\ast}_{\lfloor n/k \rfloor}), \ldots,f(y^{\ast}_{\lfloor n/k \rfloor})}^{k-n+k \lfloor n/k \rfloor \text{ times }}\Big). \\
K(y_{n-1}^{\ast}) &= \Big(\overbrace{f(y^{\ast}_{\lceil (n-1)/k \rceil}), \ldots,f(y^{\ast}_{\lceil (n-1)/k \rceil})}^{n-1 - k \lfloor (n-1)/k \rfloor \text{ times }}, \overbrace{f(y^{\ast}_{\lfloor(n-1)/k \rfloor}), \ldots,f(y^{\ast}_{\lfloor (n-1)/k \rfloor})}^{k-n+1+k \lfloor (n-1)/k \rfloor \text{ times }}\Big).
\end{align*}
We have two cases: (i) $\lceil \frac{n}{k} \rceil > \lceil \frac{n-1}{k} \rceil$, or $n \equiv 1 \pmod k$, and (ii) $\lceil \frac{n}{k} \rceil = \lceil \frac{n-1}{k} \rceil$, or $n \not\equiv 1 \pmod k$. 

\textit{Case (i): $\lceil \frac{n}{k} \rceil > \lceil \frac{n-1}{k} \rceil$}. This case requires $n \geq 3$, from which $\lceil \frac{n}{k} \rceil \leq \frac{n}{k} + 1 < n$. We apply the inductive hypothesis $f(y^{\ast}_{\lceil n/k \rceil}) > f(y^{\ast}_{\lceil (n-1)/k \rceil})$ to deduce that $K(y_{n-1}^{\ast}) <_D K(y_n^{\ast})$, so $f(y_{n-1}^{\ast}) < f(y_n^{\ast})$ by~\cref{lemma:up_to_k_tree_order}. 

\textit{Case (ii): $\lceil \frac{n}{k} \rceil = \lceil \frac{n-1}{k} \rceil$}. We must have $\lfloor \frac{n-1}{k} \rfloor \neq \lceil \frac{n-1}{k} \rceil$, as otherwise $k$ must divide $n-1$, or equivalently, $n \equiv 1 \pmod k$, and we are in case (i) rather than case (ii). Therefore, $\lfloor \frac{n-1}{k} \rfloor < \lceil \frac{n-1}{k} \rceil = \lceil \frac{n}{k} \rceil < n$. Hence, although $K(y_{n-1}^{\ast})$ and $K(y_n^{\ast})$ agree in their first $n - 1 - k \lfloor \frac{n-1}{k} \rfloor$ terms, at entry $n - k \lfloor \frac{n}{k} \rfloor$, $K(y_n^{\ast})$ has $f(y^{\ast}_{\lceil n/k \rceil})$, whereas $K(y_{n-1}^{\ast})$ has $f(y^{\ast}_{\lfloor (n-1)/k \rfloor})$. Because $\lfloor \frac{n-1}{k} \rfloor < \lceil \frac{n}{k} \rceil < n$, the inductive hypothesis yields $f(y^{\ast}_{\lceil n/k \rceil}) > f(y^{\ast}_{\lfloor (n-1)/k \rfloor})$. Hence, $K(y_n^{\ast}) >_D K(y_{n-1}^{\ast})$, so $ f(y_n^{\ast}) > f(y_{n-1}^{\ast})$ by~\cref{lemma:up_to_k_tree_order}.
\end{proof}

\medskip
Using~\cref{prop:up_to_k_minimal_increasing}, we can show that $y_n = y_n^{\ast}$: the at-most-$k$-furcating tree with minimal rank on $n$ leaves is $y_n^{\ast}$. The proof is in Appendix~\ref{pf:at_most_min_equals_ast}.

\begin{theorem}
For $n \geq 1$, the at-most-$k$-furcating tree $y_n$ with minimal rank is $y_n = y_n^{\ast}$. 
\label{thm:at_most_min_equals_ast}
\end{theorem} 

We finish this section by deriving a recurrence relation for $A_n$. In particular, using \cref{eq:rank_at_most_k} and multiple applications of the hockey-stick identity $\sum_{i=0}^r \binom{x+i}{i} = \binom{x+r+1}{r}$~\cite[eq.~5.9]{GrahamEtAl94}, we have
\begin{align}
A_n & = f(y_n) = f(y_n^{\ast}) = -A_{\lceil \frac{n}{k} \rceil} + 1 + \sum_{i=1}^{k + k \lfloor \frac{n}{k} \rfloor - n} \binom{A_{\lfloor \frac{n}{k} \rfloor} + i - 1}{i} + \sum_{i=k + k \lfloor \frac{n}{k} \rfloor - n+1}^{k} \binom{A_{\lceil \frac{n}{k} \rceil} + i - 1}{i} \nonumber \\
&= -A_{\lceil \frac{n}{k} \rceil} + 1 + \bigg[\binom{A_{\lfloor \frac{n}{k} \rfloor} + k + k \lfloor \frac{n}{k} \rfloor - n}{k + k \lfloor \frac{n}{k} \rfloor - n} - 1 \bigg] + \bigg[\sum_{i=0}^{k} \binom{A_{\lceil \frac{n}{k} \rceil} + i - 1}{i} - \sum_{i=0}^{k + k \lfloor \frac{n}{k} \rfloor - n} \binom{A_{\lceil \frac{n}{k} \rceil} + i - 1}{i} \bigg] \nonumber \\ 
\label{eq:recurrence_An}
&= -A_{\lceil \frac{n}{k} \rceil} + \binom{A_{\lfloor \frac{n}{k} \rfloor} + k + k \lfloor \frac{n}{k} \rfloor - n}{k + k \lfloor \frac{n}{k} \rfloor - n} + \binom{A_{\lceil \frac{n}{k} \rceil} + k}{k} - \binom{A_{\lceil \frac{n}{k} \rceil} + k + k\lfloor\frac{n}{k} \rfloor - n}{k + k \lfloor\frac{n}{k} \rfloor - n}, 
\end{align}
where $A_1 = 1$. Note that if $k | n$, then the second and third binomial coefficients cancel to give the simpler relation 
\begin{equation}
A_n = -A_{\lceil n/k \rceil} + \binom{A_{\lfloor n/k \rfloor} + k + k \lfloor n/k \rfloor - n}{k + k \lfloor n/k \rfloor - n}.
\label{eq:simpler}
\end{equation}
If $k=2$, then the recurrence~\cref{eq:recurrence_An} is the same as setting $k=2$ in recurrence~\cref{eq:recurrence_an}, owing to the fact that an at-most-2-furcating tree is strictly bifurcating. 

For $k=3$, the trees $y_n$ appear in Table~\ref{tab:Yn-yn-trees} for small values of $n$. The first terms of $\{A_n\}$ for small values of $k$ appear in~\cref{tab:An}. The growth of $\{A_n\}$ is considerably slower than the growth of $\{B_n\}$.

\begin{table}[tb]
\centering
\small
\begin{tabular}{|c | c | c | c | c|}
\hline
    & \multicolumn{4}{c|}{$k$} \\ \cline{2-5}
$n$ &    2 &    3  &  4 &   5 \\ 
\hline
1  &   1 &   1 &   1 &   1 \\
2  &   2 &   2 &   2 &   2 \\
3  &   3 &   3 &   3 &   3 \\
4  &   4 &   5 &   4 &   4 \\
5  &   6 &   7 &   7 &   5 \\
6  &   7 &   8 &  10 &   9 \\
7  &  10 &  13 &  12 &  13 \\
8  &  11 &  16 &  13 &  16 \\
9  &  20 &  17 &  22 &  18 \\
10 &  22 &  40 &  28 &  19 \\
11 &  28 &  49 &  31 &  33 \\
12 &  29 &  51 &  32 &  43 \\
13 &  53 &  98 &  51 &  49 \\
14 &  56 & 111 &  61 &  52 \\
15 &  66 & 113 &  65 &  53 \\
16 &  67 & 148 &  66 &  87 \\
17 & 202 & 156 & 238 & 107 \\
18 & 211 & 157 & 302 & 117 \\
19 & 252 & 487 & 320 & 121 \\
20 & 254 & 542 & 323 & 122 \\
\hline
\end{tabular}
\vspace{-.2cm}
\caption{The minimal rank $A_n$ of at-most-$k$-furcating trees with $n$ leaves, as obtained by~\cref{eq:recurrence_An}. For $k=2$, the values of $A_n$ follow sequence A354970 in the On-Line Encyclopedia of Integer Sequence.}
\label{tab:An}
\end{table}

\subsection{Asymptotics} 
\label{subsec:at_most_k_asymptotics}

We now analyze the asymptotic growth of the minimal rank $A_n$ and maximal rank $B_n$. Analogously to~\cref{prop:minimal_rank_maximal_rank_related}, we establish a link between the minimal rank $A_{k^n}$ among at-most-$k$-furcating trees with $k^n$ leaves and the maximal rank $B_n$ among at-most-$k$-furcating trees with $n$ leaves.

\begin{proposition}
\label{prop:minimal_rank_maximal_rank_related46}
$A_{k^n} + 1= B_{n+2}$ for $n \geq 0$.
\end{proposition}
\begin{proof}
We induct on $n$. For the base case of $n=0$, $A_1+1=B_2$, as $A_1=1$ and $B_2=2$.

For the inductive hypothesis, given $n \geq 1$, suppose $A_{k^{n-1}} + 1 = B_{n+1}$. By \cref{eq:recurrence_Bn}, \cref{eq:recurrence_An}, and the inductive hypothesis, 
\begin{align*}
A_{k^n} + 1 &= -A_{k^{n-1}} + \binom{A_{k^{n-1}} + k}{k} + 1 = -(B_{n + 1} - 1) + \binom{B_{n + 1} - 1 + k}{k} + 1 \\ 
&= 2-B_{n+1} + \binom{B_{n + 1} + k - 1}{k} = B_{n+2}.
\end{align*}
The induction is complete.
\end{proof}

\medskip
We now obtain the asymptotic growth of $B_n$. The growth of $\{B_n\}$ has the same form as the growth of $\{b_n\}$ except with a different base. The base $\gamma_k$ depends on $k$.

\begin{theorem} 
\label{thm:asymptotic_B_n}
$B_n \sim (k!)^{\frac{1}{k-1}} \gamma_k^{(k^n)}$ for a constant $\gamma_k$ depending only on $k \geq 2$.
\end{theorem}
\begin{proof}
Rewrite \cref{eq:recurrence_Bn} as $D_n = -D_{n-1} + \binom{D_{n-1}+k}{k}$, where $D_n = B_n-1$ for $n \geq 2$. Because $B_n \sim D_n$, it suffices to show that $D_n \sim (k!)^{\frac{1}{k-1}} \gamma_k^{(k^n)}$. A similar analysis as in the proof of~\cref{thm:asymptotic_b_n} can be performed. In particular, the recurrence can be written: 
\begin{align*}
D_n &= -D_{n-1} + \frac{1}{k!} \big[\big(D_{n-1} + k\big)\big(D_{n-1} + (k-1)\big) \ldots \big(D_{n-1} + 1\big)\big] = \frac{1}{k!}[D_{n-1}^{k} + Q_k(D_{n-1})],
\end{align*}
where $Q_k(x) = (x+1)(x+2)\cdots (x+k) - x^k - (k!)x$ is a polynomial with positive integer coefficients and degree $k-1$. 
Taking logarithms and setting $Y_n = \log D_n$, for $n \geq 3$:
\begin{align*}
Y_n - k Y_{n-1} = - \log(k!) + \kappa_{n-1},
\end{align*}
where $\kappa_{n-1} = \log [1 + Q_k(D_{n-1}) / D_{n-1}^k ]$. The quantity $\kappa_{n}$ is strictly decreasing and $\kappa_n \to 0$, as the polynomial $Q_k$ has degree less than $k$ and $D_n$ grows without bound.

The substitution
$Z_n = Y_n - \log(k!) / (k-1)$ removes the $\log(k!)$: 
\begin{align*}
Z_n = kZ_{n-1} + \kappa_{n-1}.
\end{align*} 
By the method of Aho \& Sloane~\cite{AhoSloane}, we obtain the same formula as \cref{eq:aho_sloane_recurrence_strictly_k}, replacing $z_3$ by $Z_3$ and $\rho_i$ by $\kappa_i$:
\begin{align*} 
\label{eq:aho_sloane_recurrence_at_most_k}
& Z_n = k^{n-3}Z_3 + \sum_{i=3}^{n-1} k^{n - i - 1} \kappa_{i} = k^n \bigg[Z_3k^{-3} + \sum_{i=3}^{\infty} k^{-(i+1)} \kappa_{i} \bigg] - \sum_{i=n}^{\infty} k^{n-1-i} \kappa_{i}.
\end{align*}
Taking exponents of both sides yields 
\begin{align*}
D_n \exp \bigg[ -\frac{\log(k!)}{k-1}\bigg] = e^{Z_n} & = \bigg[\exp(Z_3k^{-3}) \exp \bigg(\sum_{i=3}^{\infty} k^{-(i+1)} \kappa_i \bigg) \bigg]^{(k^n)} \exp \bigg(-\sum_{i=n}^{\infty} k^{n-1-i} \kappa_i \bigg) \\ 
&= \gamma_k^{(k^n)} \exp \bigg(-\sum_{i=n}^{\infty} k^{n-1-i} \kappa_i \bigg),
\end{align*} 
where $\gamma_k = \exp(Z_3k^{-3}) \, \exp (\sum_{i=3}^{\infty} k^{-(i+1)} \kappa_i )$. Notice that $\sum_{i=n}^{\infty} k^{n-1-i} \kappa_i \leq \kappa_n \sum_{i=n}^{\infty} k^{n-1-i} = \kappa_n / (k-1) \to 0$ because $\kappa_n$ is decreasing and has limit 0. Therefore, $D_n \sim (k!)^{\frac{1}{k-1}} \gamma_k^{(k^n)}$. 
\end{proof}

\begin{table}[tb]
\centering
\begin{tabular}{ |c|c|c| }
     \hline
     $k$ & $\gamma_k$ & $\lambda_k$\\
     \hline
     2 & $1.056528765669$ & $1.246020832983$\\
     3 & $1.025545765326$ & $1.254860390515$\\
     4 & $1.012449244990$ & $1.218911497608$\\
     5 & $1.006943308930$ & $1.188845750716$\\
     6 & $1.004260710708$ & $1.165394603276$\\
     7 & $1.002802404861$ & $1.146972413490$ \\
     \hline
    \end{tabular}
    \vspace{-.2cm}
    \captionof{table}{Values of $\gamma_k$ and $\lambda_k$, calculated using $\gamma_k = \exp(Z_3k^{-3}) \exp(\sum_{i=3}^{\infty} k^{-(i+1)} \kappa_i)$ and $\lambda_k = \gamma_k^{(k^2)}$. To calculate $\kappa_i$ and $Z_3$, we evaluate the first 12 terms $\{B_n\}$ via the recurrence~\cref{eq:recurrence_Bn}. The value $k$ denotes the maximal number of immediate descendants of an internal node; $\gamma_k$ is the base of the exponent in the growth $B_n \sim (k!)^{\frac{1}{k-1}} \gamma_k^{(k^n)}$ and $\lambda_k$ is the base of the exponent in the growth $A_{k^n} \sim (k!)^{\frac{1}{k-1}}\lambda_k^{(k^n)}$. Note that the calculation is accurate beyond 12 decimal places; we show the first 12 decimal places.}
    \label{table:gamma-lambda}
\end{table}

\medskip
The constant $\gamma_k$ can be approximated numerically for small $k$ by using the first 12 terms of the sequence $\{B_n\}$. The constants $\gamma_k$ for small $k$ appear in Table~\ref{table:gamma-lambda}. Observe that $\gamma_2 = \beta_2$ because an at-most-2-furcating tree is the same as a strictly 2-furcating tree. Also observe in the table that for $k \geq 2$, $\gamma_k > \gamma_{k+1}$; we conjecture that this result holds generally.

\begin{conjecture} 
\label{conj:decreasing-gamma}
For $k \geq 2$, the constant $\gamma_k$ satisfies $\gamma_{k+1} < \gamma_k$.
\end{conjecture}

The claim in \cref{conj:decreasing-gamma} is analogous to~\cref{lemma:beta_decreasing} for the strictly $k$-furcating case. We next state a corollary of \cref{thm:asymptotic_B_n} on the growth of $A_{k^n}$.
\begin{corollary}
For $k \geq 2$, $A_{k^n} \sim (k!)^{\frac{1}{k-1}} \lambda_k^{(k^n)}$, where $\lambda_k = \gamma_k^{(k^2)}$.
\end{corollary}
\begin{proof}
Because $A_{k^n} + 1 = B_{n+2}$ for $n \geq 0$, we have that:
\begin{align*}
A_{k^n} \sim A_{k^n} + 1 = B_{n+2} \sim (k!)^{\frac{1}{k-1}} \gamma_k^{(k^{n+2})} = (k!)^{\frac{1}{k-1}} (\gamma_k^{(k^2)})^{(k^n)} = (k!)^{\frac{1}{k-1}} \lambda^{(k^n)}.
\end{align*}
The result follows.
\end{proof}

\medskip
Values of $\lambda_k$ for small $k$ appear in Table~\ref{table:gamma-lambda}; $\lambda_2 = \alpha_2$ because an at-most-2-furcating tree is the same as a strictly 2-furcating tree. We can also observe in Table~\ref{table:gamma-lambda}, comparing to Table~\ref{table: beta-alpha}, that $\gamma_k > \beta_k$ for the values of $k \geq 3$ shown. We verify this inequality for all $k \geq 3$, showing that the maximal rank among at-most-$k$-furcating trees on $n$ leaves grows faster than the maximum among strictly $k$-furcating trees on $(n-1)(k-1)+1$ leaves. The proof is in Appendix~\ref{pf:gamma-beta-related}.
\begin{theorem} 
\label{thm:gamma-beta-comparison}
$\gamma_k \geq \beta_k$, with equality if and only if $k=2$.
\end{theorem}

\section{Discussion} 
\label{sec:discussion}

We have obtained the minimal and maximal ranks of unlabeled multifurcating rooted trees in the bijective schemes of Maranca \& Rosenberg~\cite{Maranca2024} that encode such trees with the positive integers. The results generalize corresponding results of~\cite{Rosenberg2021} on minimal and maximal ranks in the bifurcating case. 

For strictly $k$-furcating trees, $k \geq 2$, we obtained a recursive equation for the maximal (\cref{eq:recurrence_bn}) and minimal ranks among trees with $(n-1)(k-1)+1$ leaves, $n \geq 1$ (\cref{eq:recurrence_an}). We showed that asymptotically, the maximal rank grows superexponentially with $n$; the growth follows $b_n \sim (k!)^{\frac{1}{k-1}} \beta_k^{(k^n)}$ (\cref{thm:asymptotic_b_n}), where $\beta_k$ is a constant that depends on $k$~(\cref{table: beta-alpha}). The minimal rank among trees with $k^n$ leaves is one less than the maximal rank among trees with $(n+1)(k-1)+1$ leaves (\cref{prop:minimal_rank_maximal_rank_related}).

For at-most-$k$-furcating trees, we have obtained analogous results for the maximal (\cref{eq:recurrence_Bn}) and minimal ranks among trees with $n$ leaves (\cref{eq:recurrence_An}). Asymptotically, the maximal rank grows superexponentially with $n$; the growth follows $B_n \sim (k!)^{\frac{1}{k-1}} \gamma_k^{(k^n)}$ (\cref{thm:asymptotic_B_n}), where $\gamma_k$ depends on $k$~(\cref{table:gamma-lambda}). The minimal rank among trees with $k^n$ leaves is one less than the maximal rank among trees with $n+2$ leaves (\cref{prop:minimal_rank_maximal_rank_related46}).

The at-most-$k$-furcating trees include the strictly $k$-furcating trees, and we have observed numerically that for small $k \geq 3$, the growth constants satisfy $\gamma_k > \beta_k$ (\cref{table: beta-alpha}, \cref{table:gamma-lambda}), so that the growth of the number of at-most-$k$-furcating trees with $n$ leaves empirically exceeds the growth of the number of strictly $k$-furcating trees with $(n-1)(k-1)+1$ leaves for $k \geq 3$. Indeed, Theorem~\ref{thm:gamma-beta-comparison} demonstrates that $\gamma_k > \beta_k$ for $k \geq 3$. The constants $\beta_k$ decrease with $k$ (\cref{lemma:beta_decreasing}), and we conjecture that the constants $\gamma_k$ decrease with $k$ as well (\cref{conj:decreasing-gamma}).

In both types of multifurcating trees, tree ranks can serve as measures of tree balance. In particular, with a fixed number of leaves, a highly balanced tree has minimal rank, and a highly imbalanced tree has maximal rank. This phenomenon was observed by~\cite{Rosenberg2021} in the bifurcating case, so that a measure $[\log f(t) - \log a_{m(t)}] / [\log b_{m(t)} - \log a_{m(t)}]$ can serve as an index that measures increasing imbalance of a tree $t$. An update of this proposal based on Devroye et al.~\cite{Devroye2025} instead suggests the use of $[\log_2 \log f(t) - \log_2 \log a_{m(t)}] /[\log_2 \log b_{m(t)} - \log_2 \log a_{m(t)}]$, as the numerical values of the ranks of different trees are comparable after the double logarithm is taken. Modifying this suggestion for $k$-furcation, our results here suggests that a comparably sensible index in the setting of multifurcating trees is 
\begin{equation}
\label{eq:balance}
\frac{\log_k \log f(t) - \log_k \log a_{m(t)}}{\log_k \log b_{m(t)} - \log_k \log a_{m(t)}},
\end{equation}
According to this scheme, the minimally balanced tree has value 1 for the index and the maximally balanced tree has value 0.

In mathematical evolutionary biology, multifurcating trees have the potential to describe biological phenomena with rapid diversification and large numbers of offspring~\cite{Eldon20}, phenomena that that are not as easily modeled with bifurcating trees. The study adds to recent interest in the mathematical understanding of multifurcating trees in evolutionary biology~\cite{DickeyAndRosenberg25, Maranca2024, MirEtAl18, Wirtz24, ZhangAndPalacios25}, providing further information about encoding schemes for two multifurcating tree classes, strictly $k$-furcating and at-most-$k$-furcating, with potential for application in diverse types of evolutionary studies.

\vskip .3cm
\noindent {\bf Acknowledgements.} We acknowledge support from National Science Foundation grant DMS-2450005.
\vskip .3cm
\noindent {\bf Data Availability.} This article has no associated data.
\vskip .3cm
\noindent {\bf Conflicts of Interest.} The authors have no conflicts of interest to declare.


\appendix
\renewcommand{\thesubsection}{Appendix \Alph{subsection}}
\label{sec:appendix}
\counterwithin{theorem}{section}

\section{Proof of~\cref{lemma:strict_tree_order}} 
\label{pf:strict_tree_order}

To prove \cref{lemma:strict_tree_order}, we first state an elementary result that we call the ``sum-peeling lemma.'' This lemma is used repeatedly in Appendix~\ref{pf:strict_tree_order} and Appendix~\ref{pf:up_to_k_tree_order} to remove the outermost term of a nested sum. Such sums appear frequently in the proofs of~\cref{lemma:strict_tree_order} and~\cref{lemma:up_to_k_tree_order}.

\begin{lemma}[Sum-peeling lemma] 
\label{lemma: sum-peeling lemma}
If $a_k, a_{k+1}, \ldots, a_x, b_x$ are all non-negative integers, and $a_k \leq a_{k-1} \leq \ldots \leq a_{x} \leq b_x-1$, then 
\begin{align} 
\label{ineq: sum-peeling-lemma}
\sum_{y_x=a_x}^{b_{x}-1} \sum_{y_{x+1}=a_{x+1}}^{y_x} \ldots \sum_{y_k=a_k}^{y_{k-1}} 1 \geq \sum_{y_{x+1}=a_{x+1}}^{b_x-1} \sum_{y_{x+2}=a_{x+2}}^{y_{x+1}} \ldots \sum_{y_k=a_k}^{y_{k-1}} 1.
\end{align}
\end{lemma}
\begin{proof}
Because $a_x \leq b_x-1$, we have 
\begin{align*}
\sum_{y_{x} = a_x}^{b_x-1} \sum_{y_{x+1}=a_{x+1}}^{y_{x}} \ldots \sum_{y_k=a_k}^{y_{k-1}} 1 \geq \sum_{y_{x} = b_x-1}^{b_x-1} \sum_{y_{x+1}=a_{x+1}}^{y_{x}} \sum_{y_{x+2}=a_{x+2}}^{y_{x+1}} \ldots \sum_{y_k=a_k}^{y_{k-1}} 1 = \sum_{y_{x+1}=a_{x+1}}^{b_x-1} \sum_{y_{x+2}=a_{x+2}}^{y_{x+1}} \ldots \sum_{y_k=a_k}^{y_{k-1}} 1. 
\end{align*}
This completes the proof.
\end{proof}

\medskip
We are now ready for the proof of \cref{lemma:strict_tree_order}.
\medskip

\begin{proof}
For the equality condition, if $f(t_1)=f(t_2)$, then because $f$ is bijective, $t_1=t_2$, and hence $K(t_1)=K(t_2)$. If $K(t_1) = K(t_2)$, then $t_1$ and $t_2$ have the same subtrees in canonical order. Hence, $t_1=t_2$ and $f(t_1) = f(t_2)$. 
 
Next, suppose $K(t_1) <_D K(t_2)$. Let $K(t_1) = (a_1, a_2, \ldots, a_k)$ and $K(t_2) = (b_1, b_2, \ldots, b_k)$, and let $x$ be the smallest index with $a_x \neq b_x$. Because $K(t_1) <_D K(t_2)$, we must have $a_x < b_x$. By definition (\cref{eq:rank_strictly_k}),
\begin{align*}
f(t_2) &= 2 + \binom{b_1 + k-2}{k} + \ldots + \binom{b_{x-1} + k - x}{k- x+2} + \binom{b_x + k-x-1}{k-x+1} +\ldots + \binom{b_k+1-2}{1} \\ 
&= 2 + \binom{a_1 + k-2}{k} + \ldots + \binom{a_{x-1} + k - x}{k- x+2} + \binom{b_x + k-x-1}{k-x+1} + \ldots + \binom{b_k+1-2}{1} \\ 
& \geq 2 + \binom{a_1 + k-2}{k} + \ldots + \binom{a_{x-1} + k - x}{k- x+2} + \binom{b_x + k -x - 1}{k - x + 1}.
\end{align*}

For positive integers $y_n, y_{n+1}, \ldots, y_k$ bounded above by positive integer $y_{n-1}$, Lemma 4.1 of~\cite{Maranca2024} proves 
\begin{align} 
\label{eq:strictly_k_binomial_to_sum}
\binom{y_{n-1} + k - n}{k - n + 1} = \sum_{y_n=1}^{y_{n-1}} \sum_{y_{n+1}=1}^{y_n} \ldots \sum_{y_k=1}^{y_{k-1}} 1.
\end{align}
Applying~\cref{eq:strictly_k_binomial_to_sum} with the choice of $n=x$ and $y_{n-1} = b_x - 1$, 
\begin{align}
\binom{b_x + k -x - 1}{k - x + 1} &= 
\bigg( \sum_{y_{x} = 1}^{a_x-1} \sum_{y_{x+1}=1}^{y_{x}} \ldots \sum_{y_k=1}^{y_{k-1}} 1 \bigg) + \bigg( \sum_{y_{x} = a_x}^{b_x-1} \sum_{y_{x+1}=1}^{y_{x}} \ldots \sum_{y_k=1}^{y_{k-1}} 1 \bigg) \nonumber \\ 
&=\binom{a_x+k-x-1}{k-x+1} + \sum_{y_{x} = a_x}^{b_x-1} \sum_{y_{x+1}=1}^{y_{x}} \ldots \sum_{y_k=1}^{y_{k-1}} 1. 
\label{eq:binomial-sum-strictly-k} 
\end{align}

Applying the sum-peeling inequality in \cref{ineq: sum-peeling-lemma} to the sum in~\cref{eq:binomial-sum-strictly-k}, we have
\begin{align} 
\sum_{y_{x} = a_x}^{b_x-1} \sum_{y_{x+1}=1}^{y_{x}} \ldots \sum_{y_k=1}^{y_{k-1}} 1 
& \geq \sum_{y_{x+1} = 1}^{b_x-1} \sum_{y_{x+2}=1}^{y_{x+1}} \ldots \sum_{y_k=1}^{y_{k-1}} 1 \nonumber \\ 
&= \bigg( \sum_{y_{x+1}= 1}^{a_{x+1} - 1} \sum_{y_{x+2}=1}^{y_{x+1}} \ldots \sum_{y_k=1}^{y_{k-1}} 1 \bigg) + \bigg( \sum_{y_{x+1}= a_{x+1}}^{b_x-1} \sum_{y_{x+2}=1}^{y_{x+1}} \ldots \sum_{y_k=1}^{y_{k-1}} 1 \bigg). \nonumber
\end{align}
We then apply~\cref{eq:strictly_k_binomial_to_sum} to the first sum and \cref{ineq: sum-peeling-lemma} to the second, obtaining
\begin{align} 
 \sum_{y_{x} = a_x}^{b_x-1} \sum_{y_{x+1}=1}^{y_{x}} \ldots \sum_{y_k=1}^{y_{k-1}} 1 &= \binom{a_{x+1} + k - x - 2}{k - x } + \sum_{y_{x+1}= a_{x+1}}^{b_x-1} \sum_{y_{x+2}=1}^{y_{x+1}} \ldots \sum_{y_k=1}^{y_{k-1}} 1 \nonumber \\ 
& \geq \binom{a_{x+1} + k - x - 2}{k - x } + \sum_{y_{x+2}=1}^{b_{x} - 1} \sum_{y_{x+3}=1}^{y_{x+2}} \ldots \sum_{y_k=1}^{y_{k-1}} 1. \nonumber
\end{align}
Plugging back into~\cref{eq:binomial-sum-strictly-k},
\begin{align} 
\label{eq: split-peel-argument-two-times}
\binom{b_x+k-x-1}{k-x+1} \geq \binom{a_x+k-x-1}{k-x+1} + \binom{a_{x+1}+k-x-2}{k-x} + \sum_{y_{x+2}=1}^{b_{x} - 1} \sum_{y_{x+3}=1}^{y_{x+2}} \ldots \sum_{y_k=1}^{y_{k-1}} 1.
\end{align}
Splitting the left hand side of~\cref{eq:binomial-sum-strictly-k} twice and applying the sum-peeling-lemma twice has given us~\cref{eq: split-peel-argument-two-times}. We apply this ``split-and-peel'' argument $k-x-2$ more times, $k-x$ in total, leaving the summation
\begin{align}
\binom{b_x + k -x - 1}{k - x + 1} & \geq \binom{a_x+k-x-1}{k-x+1} + \binom{a_{x+1} + k - x - 2}{k - x } + \ldots + \binom{a_{k-1}}{2} + \sum_{y_k=1}^{b_{x} - 1} 1 \nonumber \\ 
& > \binom{a_x+k-x-1}{k-x+1} + \binom{a_{x+1} + k - x - 2}{k - x } + \ldots + \binom{a_{k-1}}{2} + \sum_{y_k=1}^{a_{k} - 1} 1 \nonumber \\ 
& = \binom{a_x+k-x-1}{k-x+1} + \binom{a_{x+1} + k - x - 2}{k - x } + \ldots + \binom{a_{k-1}}{2} + \binom{a_k - 1}{1}. 
\label{ineq:last-inequality-strictly-k}
\end{align}
In the second-to-last step, we have used the fact that $a_x < b_x$ and $a_k \leq a_x$. Therefore, by inequality (\ref{ineq:last-inequality-strictly-k}), 
\begin{align*}
f(t_2) &= 2 + \binom{a_1 + k-2}{k} + \ldots + \binom{a_{x-1} + k - x}{k- x+2} + \binom{b_x + k -x - 1}{k - x + 1} +\ldots + \binom{b_k+1-2}{1} \\ 
& \geq 2 + \binom{a_1 + k-2}{k} + \ldots + \binom{a_{x-1} + k - x}{k- x+2} + \binom{b_x + k -x - 1}{k - x + 1} \\ 
& > 2 + \binom{a_1 + k-2}{k} + \ldots + \binom{a_{x-1} + k - x}{k- x+2} \\ 
& \quad + \left[\binom{a_x+k-x-1}{k-x+1} + \binom{a_{x+1} + k - x - 2}{k - x } + \ldots + \binom{a_{k-1}}{2} + \binom{a_k - 1}{1} \right] \\ 
&= f(t_1).
\end{align*}

For the converse, suppose that $f(t_1) < f(t_2)$. Assume for the sake of contradiction that $K(t_1) \geq_D K(t_2)$. The assumption $K(t_1) \leq_D K(t_2)$ in the argument above can be replaced by $K(t_2) \leq_D K(t_1)$, in which case the conclusion will be $f(t_2) \leq f(t_1)$. But this statement contradicts the assumption $f(t_1) < f(t_2)$, meaning that in fact, $K(t_1) <_D K(t_2)$. 
\end{proof}

\section{Proof of~\cref{lemma:up_to_k_tree_order}}
\label{pf:up_to_k_tree_order}

We prove~\cref{lemma:up_to_k_tree_order}, the at-most-$k$-furcating counterpart of~\cref{lemma:strict_tree_order} demonstrated in Appendix~\ref{pf:strict_tree_order}. The proof is another application of~\cref{ineq: sum-peeling-lemma}. Some binomial identities are needed for the proof. One property that quickly follows from Lemma 5.1 of~\cite{Maranca2024} is that 
\begin{align} 
\label{eq:binom_coeff_at_most_k}
& \sum_{y_x = 1}^{b_x-1} \sum_{y_{x+1} = 1}^{y_{x}} \sum_{y_{x+2} = 0}^{y_{x+1}} \ldots \sum_{y_k=0}^{y_{k-1}} 1 = \binom{b_x + k -x}{k - x + 1} - b_x.
\end{align} 
Noting that $b_x$ terms $(y_x, y_{x+1}, \ldots, y_k)$ have the form $(\ast, 0, \ldots, 0),$ where $\ast$ is an integer in $[0, b_x-1]$,~\cref{eq:binom_coeff_at_most_k} follows from the equation in Lemma 5.1 of~\cite{Maranca2024}, 
\begin{align} 
\label{eq:sum_to_binom_at_most_k}
& \sum_{y_x = 0}^{b_x-1} \sum_{y_{x+1} = 0}^{y_{x}} \sum_{y_{x+2} = 0}^{y_{x+1}} \ldots \sum_{y_k=0}^{y_{k-1}} 1 = \binom{b_x + k -x}{k - x + 1},
\end{align}
noting that 
\begin{align} 
\label{eq:second_sum_to_binom_at_most_k}
& \sum_{y_x = 1}^{b_x-1} \sum_{y_{x+1} = 0}^{y_{x}} \sum_{y_{x+2} = 0}^{y_{x+1}} \ldots \sum_{y_k=0}^{y_{k-1}} 1 = \binom{b_x + k -x}{k - x + 1} - 1.
\end{align}

The following inequality is the workhorse of the proof of Lemma~\ref{lemma:up_to_k_tree_order}.
\begin{lemma} 
\label{lemma: appendix-b-helper-lemma}
If $a_k, a_{k+1}, \ldots, a_x, b_x$ are all positive integers, and $a_k \leq a_{k-1} \leq \ldots \leq a_x \leq b_x-1$, then 
\begin{align} 
\label{ineq: binom-to-sum-comparison}
\binom{b_x + k - x}{k - x + 1} > \sum_{i=1}^{k-x+1} \binom{a_{k-i+1}+i-1}{i}.
\end{align}
\end{lemma}
\begin{proof}
We appeal to a similar strategy as in the proof of~\cref{lemma:strict_tree_order} in Appendix~\ref{pf:strict_tree_order}. In particular, we split the summation of~\cref{eq:sum_to_binom_at_most_k} into two terms. The sum-peeling lemma (\cref{lemma: sum-peeling lemma}) eliminates the outer sum of one of the terms, and~\cref{eq:sum_to_binom_at_most_k} simplifies the other. This ``split-and-peel'' argument gives 
\begin{align}
\binom{b_x + k - x}{k - x + 1}
&= \bigg( \sum_{y_x=a_x}^{b_x - 1} \sum_{y_{x+1}=0}^{y_x} \ldots \sum_{y_k=0}^{y_{k-1}}1 \bigg) + \bigg( \sum_{y_x=0}^{a_x - 1} \sum_{y_{x+1}=0}^{y_x} \ldots \sum_{y_k=0}^{y_{k-1}}1 \bigg) \nonumber \\
& \stackrel{(\ref{eq:sum_to_binom_at_most_k})}{=} \bigg( \sum_{y_x=a_x}^{b_x - 1} \sum_{y_{x+1}=0}^{y_x} \ldots \sum_{y_k=0}^{y_{k-1}}1 \bigg) + \binom{a_x+k-x}{k-x+1} \nonumber \\
& \stackrel{(\ref{ineq: sum-peeling-lemma})}{\geq} \bigg( \sum_{y_{x+1}=0}^{b_x-1} \sum_{y_{x+2}=0}^{y_{x+1}} \ldots \sum_{y_k=0}^{y_{k-1}} 1 \bigg) + \binom{a_x+k-x}{k-x+1}. 
\label{eqn: appendix-b-split-and-peel}
\end{align} 
Applying the split-and-peel argument to~\cref{eqn: appendix-b-split-and-peel}, we obtain 
\begin{align}
& \bigg( \sum_{y_{x+1}=0}^{b_x-1} \sum_{y_{x+2}=0}^{y_{x+1}} \ldots \sum_{y_k=0}^{y_{k-1}} 1 \bigg) + \binom{a_x+k-x}{k-x+1} \nonumber \\
&= \bigg( \sum_{y_{x+1}=a_{x-1}}^{b_x - 1} \sum_{y_{x+2}=0}^{y_{x+1}} \ldots \sum_{y_k=0}^{y_{k-1}} 1 \bigg) + \bigg( \sum_{y_{x+1}=0}^{a_{x-1}-1} \sum_{y_{x+2}=0}^{y_{x+1}} \ldots \sum_{y_k=0}^{y_{k-1}} 1 \bigg) + \binom{a_x+k-x}{k-x+1} \nonumber \\ 
& \stackrel{(\ref{eq:sum_to_binom_at_most_k})}{=} \bigg( \sum_{y_{x+1}=a_{x-1}}^{b_x - 1} \sum_{y_{x+2}=0}^{y_{x+1}} \ldots \sum_{y_k=0}^{y_{k-1}} 1 \bigg) + \binom{a_{x+1}+k-x-1}{k-x} + \binom{a_x+k-x}{k-x+1} \nonumber \\
& \stackrel{(\ref{ineq: sum-peeling-lemma})}{\geq} \bigg( \sum_{y_{x+2}=0}^{b_x-1} \sum_{y_{x+3}=0}^{y_{x+2}} \ldots \sum_{y_k=0}^{y_{k-1}} 1 \bigg) + \binom{a_{x+1}+k-x-1}{k-x} + \binom{a_x+k-x}{k-x+1}. 
\label{eqn: appendix-b-second-sum-and-peel}
\end{align}

Each application of the split-and-peel argument removes one layer of the summation and adds a new binomial term. Because $k-x-1$ summation layers remain in~\cref{eqn: appendix-b-second-sum-and-peel}, applying the split-and-peel method $k-x-2$ more times, $k-x$ times in total, gives 
\begin{align*}
& \bigg( \sum_{y_{x+2}=0}^{b_x-1} \sum_{y_{x+3}=0}^{y_{x+2}} \ldots \sum_{y_k=0}^{y_{k-1}} 1 \bigg) + \binom{a_{x+1}+k-x-1}{k-x} + \binom{a_x+k-x}{k-x+1} \\
& \geq \ldots \geq \bigg( \sum_{y_k=0}^{b_x-1} 1 \bigg) + \sum_{i=2}^{k-x+1} \binom{a_{k-i+1} + i - 1}{i} \\
&> \Big( \sum_{y_k=0}^{a_k-1} 1\Big) + \sum_{i=2}^{k-x+1} \binom{a_{k-i+1} + i - 1}{i} \\
&= a_k + \sum_{i=2}^{k-x+1} \binom{a_{k-i+1} + i - 1}{i} \\
&= \sum_{i=1}^{k-x+1} \binom{a_{k-i+1} + i - 1}{i},
\end{align*}
where the last inequality follows from $a_k \leq a_x < b_x$.
\end{proof}

\medskip

We are now ready for the proof of Lemma~\ref{lemma:up_to_k_tree_order}.

\medskip

\begin{proof}
The equality case is clear: if $f(t_1)=f(t_2)$, then by the bijectivity of $f$, $t_1=t_2$, so that $K(t_1)=K(t_2)$. If $K(t_1) = K(t_2)$, then $t_1$ and $t_2$ have the same trees in canonical order and are the same. Hence $f(t_1)=f(t_2)$. 

Suppose that $K(t_1) <_D K(t_2)$. Write $K(t_1) = (a_1, a_2, \ldots, a_k)$ and $K(t_2) = (b_1, b_2, \ldots, b_k)$, and let $x$ be the smallest index at which $a_x \neq b_x$. Because $K(t_1) <_D K(t_2)$, it follows that $a_x < b_x$. We divide into two cases based on the value of $x$.

\textit{Case (i): $x \geq 2$}. Because $x \geq 2$, it follows that $a_1=b_1$. Beginning from \cref{eq:rank_at_most_k}, 
\begin{align*}
f(t_2) &= -b_1 + 1 + \sum_{i=1}^{k} \binom{b_{k-i+1} + i -1}{i} \\
&= -a_1 + 1 + \sum_{i=1}^{k-x+1} \binom{b_{k-i+1} + i-1}{i} + \sum_{i=k-x+2}^{k} \binom{a_{k-i+1} + i-1}{i} \\ 
& \geq -a_1 + 1 + \binom{b_x + k - x}{k - x + 1} + \sum_{i=k-x+2}^{k} \binom{a_{k-i+1} + i-1}{i},
\end{align*} 
where in the last step we have extracted a single term from the summation, corresponding to $i=k-x+1$.

Now, by \cref{ineq: binom-to-sum-comparison}, because $a_k \leq a_{k-1} \leq \ldots \leq a_{x+1} \leq a_x < b_x$ by assumption,
\begin{align*}
& -a_1 + 1 + \binom{b_x+k-x}{k-x+1} + \sum_{i=k-x+2}^{k} \binom{a_{k-i+1}+i-1}{i} \\
& > -a_1 + 1 + \sum_{i=1}^{k-x+1} \binom{a_{k-i+1}+i-1}{i} + \sum_{i=k-x+2}^{k} \binom{a_{k-i+1}+i-1}{i} \\
&= -a_1 + 1 + \sum_{i=1}^{k} \binom{a_{k-i+1}+i-1}{i} = f(t_1).
\end{align*}

\textit{Case (ii): $x=1$}. Noting $b_2 \geq 1$ and applying \cref{eq:binom_coeff_at_most_k},
\begin{align}
f(t_2) &= -b_1 + 1 + \sum_{i=1}^{k} \binom{b_{k-i+1} + i -1}{i} \nonumber \\
&= \bigg[ -b_1 + \binom{b_1 + k -1}{k} + 2\bigg] + \bigg[ \binom{b_2 + k -2}{k - 1} - 1 \bigg] + \sum_{i=1}^{k-2} \binom{b_{k-i+1} + i - 1}{i} \nonumber \\
& \geq -b_1 + \binom{b_1 + k -1}{k} + 2 \nonumber \\ 
& = \bigg( \sum_{y_1=1}^{b_1-1} \sum_{y_2=1}^{y_1} \sum_{y_3=0}^{y_2} \ldots \sum_{y_k=0}^{y_{k-1}} 1 \bigg) + 2 \nonumber \\
&= \bigg( \sum_{y_1=1}^{a_1-1} \sum_{y_2=1}^{y_1} \sum_{y_3=0}^{y_2} \ldots \sum_{y_k=0}^{y_{k-1}} 1 \bigg) + \bigg( \sum_{y_1=a_1}^{b_1-1} \sum_{y_2=1}^{y_1} \sum_{y_3=0}^{y_2} \ldots \sum_{y_k=0}^{y_{k-1}} 1 \bigg) + 2.
\label{eq: appendix-b-simplify-terms}
\end{align}
We simplify the leftmost summation of~\cref{eq: appendix-b-simplify-terms} using~\cref{eq:binom_coeff_at_most_k}. We simplify the rightmost summation by removing its outermost layer using the sum-peeling lemma (\cref{lemma: sum-peeling lemma}). Because $a_2 \leq a_1 < b_1$ in this case,
\begin{align}
f(t_2) & \stackrel{(\ref{ineq: sum-peeling-lemma}),(\ref{eq:binom_coeff_at_most_k})}{\geq} -a_1 + \binom{a_1 + k - 1}{k} + \bigg( \sum_{y_2=1}^{b_1-1} \sum_{y_3=0}^{y_2} \ldots \sum_{y_k=0}^{y_{k-1}} 1 \bigg) + 2 \nonumber \\
&= -a_1 + \binom{a_1 + k - 1}{k} + \bigg( \sum_{y_2=1}^{a_2-1} \sum_{y_3=0}^{y_2} \ldots \sum_{y_k=0}^{y_{k-1}} 1 \bigg) + \bigg( \sum_{y_2=a_2}^{b_1-1} \sum_{y_3=0}^{y_2} \ldots \sum_{y_k=0}^{y_{k-1}} 1 \bigg) + 2. 
\label{eqn: appendix-b-simplify-terms-part-ii} 
\end{align}

We next simplify the leftmost summation of~\cref{eqn: appendix-b-simplify-terms-part-ii} using~\cref{eq:second_sum_to_binom_at_most_k}. We simplify the rightmost summation by removing its outermost layer using the sum-peeling lemma. The peeling step produces a form for which~\cref{eq:sum_to_binom_at_most_k} applies, permitting application of~\cref{lemma: appendix-b-helper-lemma}, which applies because $a_k \leq \ldots a_4 \leq a_3 \leq b_1-1$.
\begin{align*}
f(t_2) & \stackrel{(\ref{ineq: sum-peeling-lemma}),(\ref{eq:second_sum_to_binom_at_most_k})}{\geq} -a_1 + \binom{a_1 + k - 1}{k} + \bigg[ \binom{a_2 + k - 2}{k - 1} - 1 + \bigg( \sum_{y_3=0}^{b_1 - 1} \sum_{y_4=0}^{y_3} \ldots \sum_{y_k=0}^{y_{k-1}} 1 \bigg) \bigg] + 2 \\ 
& \stackrel{(\ref{eq:sum_to_binom_at_most_k})}{=} -a_1 + \binom{a_1 + k - 1}{k} + \binom{a_2 + k - 2}{k - 1} - 1 + \binom{b_1+k-3}{k-2} + 2 \\
& \stackrel{(\ref{ineq: binom-to-sum-comparison})}{>} -a_1 + \binom{a_1 + k - 1}{k} + \binom{a_2 + k - 2}{k - 1} - 1 + \bigg[ \sum_{i=1}^{k-2} \binom{a_{k-i+1} + i -1}{i} \bigg] + 2 \\ 
&= -a_1 + 1 + \sum_{i=1}^{k} \binom{a_{k-i+1} + i - 1}{i} = f(t_1).
\end{align*}

With both cases established, to prove the converse --- namely that if $f(t_1) < f(t_2)$, then $K(t_1) <_D K(t_2)$ --- we show the contrapositive. Suppose that $K(t_1) \geq_D K(t_2)$. Replacing the assumption $K(t_1) \leq_D K(t_2)$ in the argument above by $K(t_2) \leq_D K(t_1)$, we conclude $f(t_2) \leq f(t_1)$, verifying the contrapositive. 
\end{proof}

\section{{Proof of $2 + \binom{x+k-2}{k} > x$ for all positive integers $x \geq 0$ and $k \geq 2$}}
\label{sec:proof_binomial_ineq}

\begin{proof}
The inequality is easily verified for $x = 0$, 1, and 2. Suppose $x \geq 3$. Then
\begin{align*}
    \binom{x-2+k}{k} &= \frac{(x-2+k)(x-3+k) \cdots x(x-1)}{k!} = \prod_{i=1}^{k} \left( \frac{x-2+i}{i} \right) = (x-1) \prod_{i=2}^{k} \left( \frac{x-2+i}{i} \right) \\ & \geq (x-1) \prod_{i=2}^{k} \left( \frac{3-2+i}{i} \right) = (x-1) \prod_{i=2}^{k} \left( \frac{i+1}{i} \right) > x-1 > x-2.
    \end{align*}
    Adding 2 to both sides yields the desired inequality.
\end{proof}

\section{Proof of~\cref{thm:strictly_min_equals_ast}}
\label{pf:strictly_min_equals_ast}

We prove that $z_n = z_n^{\ast}$, establishing that the constructed tree $z_n^{\ast}$ is the minimal-rank strictly $k$-furcating tree with $(n-1)(k-1)+1$ leaves. We first establish a simple statement: if an element of a list sorted in descending order is replaced by a smaller element, then lexicographically, the reordered list is less than the initial list.
\begin{lemma} 
\label{lemma: replacement-lemma}
For real numbers $a_1, a_2, \ldots, a_k$ with $a_1 \geq a_2 \geq \ldots \geq a_k$, Suppose the $i$th entry of $A = (a_1, a_2, \ldots, a_k)$ is replaced by $b < a_i$. If $B$ is the resulting $k$-tuple sorted in descending order, then $B <_D A$.
\end{lemma}
\begin{proof}
Because $a_1 \geq a_2 \geq \ldots \geq a_i > b$, $A$ and $B$ agree in the first $i-1$ coordinates. Therefore, the first coordinate at which $A$ and $B$ differ---where the lexicographic order is determined---is at least $i$. 

Suppose $b < a_k$. Then $B = (a_1, a_2, \ldots, a_{i-1}, a_{i+1}, a_{i+2}, \ldots, a_{k}, b)$. By definition, $a_i \geq a_{i+1} \geq a_{i+2} \geq \ldots \geq a_k$ is in descending order, and if one of the inequalities is strict, then $A >_D B$. Otherwise, if $a_i=a_{i+1}=a_{i+2}=\ldots=a_k$, then in the $k$th coordinate, $a_k > b$ implies $A >_D B$.

Otherwise, let $j \leq k$ be the smallest index such that $a_{j-1} > b \geq a_j$. In this case, 
\begin{align*}
B = (a_1, a_2, \ldots, a_{i-1}, a_{i+1}, a_{i+2}, \ldots, a_{j-1}, b, a_{j}, a_{j+1}, \ldots, a_k).
\end{align*}
Once again, the sequence $a_i \geq a_{i+1} \geq \ldots \geq a_{j-1}$ is in descending order, and if one of those inequalities is strict,  then $A >_D B$. Otherwise, if $a_i = a_{i+1}=\ldots=a_{j-1}$, then the $(j-1)$-th coordinate satisfying $a_{j-1} > b$ implies $A >_D B.$
\end{proof}

\medskip
We now provide the proof of \cref{thm:strictly_min_equals_ast}.
\medskip

\begin{proof} 
We induct on $n$. The base case $n=1$ is trivial, as only one tree has a single leaf: $z_1 = z_1^{\ast}$. 

For the inductive hypothesis, suppose that $z_k = z_k^{\ast}$ for each $k$ with $1 \leq k \leq n-1$. Consider the tree of minimal rank $z_n$ with subtrees $j_1,j_2,\ldots,j_k$ in canonical order, so that $f(j_1) \geq f(j_2) \geq \ldots \geq f(j_k)$ and $K(z_n) = \big(f(j_1), f(j_2), \ldots, f(j_k)\big)$.

We first argue that for each $i$, $j_i$ is the tree of minimal rank for its number of leaves, $m(j_i)$, or $j_i=z^{\ast}_{(m(j_i) + k - 2) / (k-1)}$ for all $i$, $1 \leq i \leq k$. Assume, for the sake of contradiction, that for some $i$, $j_i$ and $z^{\ast}_{(m(j_i) + k -2)/(k-1)}$ are distinct trees; noting that the number of leaves of $z^{\ast}_{(m(j_i) + k -2)/(k-1)}$ is $(k-1) \big(\big(m(j_i) + k -2 \big)/(k-1) - 1 \big) + 1 = m(j_i)$, they have the same number of leaves. The tree $t_i'$ defined as tree $z_n$ with subtree $j_i$ replaced by $z^{\ast}_{(m(j_i)+k-2)/(k-1)}$ is then a distinct tree from $t$. Because $f(z^{\ast}_{(m(j_i)+k-2)/(k-1)}) < f(j_i)$ by the inductive hypothesis, Lemma~\ref{lemma: replacement-lemma} implies that the lexicographic orderings satisfy $K(t_i') <_D K(z_n)$. But then $f(t_i') < f(z_n)$ by~\cref{lemma:strict_tree_order}, contradicting the rank-minimality of $z_n$ among trees with $(n-1)(k-1)+1$ leaves. We conclude that $j_i$ and $z^{\ast}_{(m(j_i)+k-2)/(k-1)}$ are the same tree for all $i$, $1 \leq i \leq k$.

Thus far, we have deduced that $K(z_n) = \big(f(z_{(m(j_1)+k-2) / (k-1)}^{\ast}), f(z_{(m(j_2)+k-2) / (k-1)}^{\ast}), \ldots, f(z_{(m(j_k)+k-2) / (k-1)}^{\ast}) \big)$, with the restriction that $\sum_{i=1}^{k} m(j_i) = (n-1)(k-1)+1$. Furthermore, by~\cref{prop:strict_min_increasing}, for $j_1,j_2,\ldots,j_k$ to be a canonical ordering, we must have that $m(j_1) \geq m(j_2) \geq \ldots \geq m(j_k)$. We claim that $K(z_n) \geq_D K(z_n^{\ast})$, which will imply that $K(z_n) = K(z_n^{\ast}),$ or equivalently, $f(z_n) = f(z_n^{\ast})$. 

Assume for contradiction that $K(z_n) <_D K(z_n^{\ast})$. We will show that $z_n$ then has strictly fewer than $m(z_n) = (n-1)(k-1)+1$ leaves. Let $x$ be the smallest index such that $K(z_n)$ differs from $K(z_n^{\ast})$. There are two possibilities for $x$.

\medskip
\noindent
\textit{Case (i): $x > n - 2 + k - k \lceil \frac{n-2}{k} \rceil $}. In this case, $j_1 = j_2 = \ldots = j_{n - 2 + k - k \lceil (n-2)/k \rceil} = z^{\ast}_{\lceil (n-2)/k \rceil + 1}$, and $j_{n - 1 + k - k \lceil (n-2)/k \rceil} = \ldots = j_{x-1} = z^{\ast}_{\lceil (n-2) / k \rceil}$. Because $j_x$ differs from $z^{\ast}_{\lceil (n-2) / k \rceil}$ and satisfies $f(j_x) < f(z^{\ast}_{\lceil (n-2) / k \rceil})$, it follows that $m(j_x) < m(z^{\ast}_{\lceil (n-2) / k \rceil})$ by~\cref{prop:strict_min_increasing}. Hence, 
\begin{align*}
m(z_n) &= \sum_{i=1}^{k} m(j_i) \\
& = \sum_{i=1}^{n-2+k-k\lceil (n-2)/k \rceil} m(z^{\ast}_{\lceil (n-2)/k \rceil+ 1} ) + \sum_{i=n-1+k-k\lceil (n-2)/k \rceil}^{x-1} m(z^{\ast}_{\lceil (n-2)/k \rceil}) + \sum_{i=x}^{k} m(j_i) \\ & \leq \sum_{i=1}^{n-2+k-k\lceil (n-2)/k \rceil} m(z^{\ast}_{\lceil (n-2)/k \rceil + 1} ) + \sum_{i=n-1+k-k\lceil (n-2)/k \rceil}^{x-1} m(z^{\ast}_{\lceil (n-2)/k \rceil}) + \sum_{i=x}^{k} m(j_x) \\ 
& < \sum_{i=1}^{n-2+k-k\lceil (n-2)/k \rceil} m(z^{\ast}_{\lceil (n-2)/k \rceil+ 1} ) + \sum_{i=n-1+k-k\lceil (n-2)/k \rceil}^{x-1} m(z^{\ast}_{\lceil (n-2)/k \rceil}) + \sum_{i=x}^{k} m(z^{\ast}_{\lceil (n-2)/k \rceil}) \\ 
&= \sum_{i=1}^{n-2+k-k\lceil (n-2) / k \rceil} m(z^{\ast}_{\lceil (n-2)/ k \rceil + 1} ) + \sum_{i=n-1+k-k\lceil (n-2) / k \rceil}^{k} m(z^{\ast}_{\lceil (n-2)/k \rceil}) \\
& = m(z_n^{\ast}) = (n-1)(k-1)+1,
\end{align*}
contradicting the requirement that $m(z_n) = (n-1)(k-1)+1$.
\hfill \break

\noindent
\textit{Case (ii): $x \leq n - 2 + k - k \lceil \frac{n-2}{k} \rceil$}. In this case, $j_1 = j_2 = \ldots = j_{x-1} = z^{\ast}_{\lceil (n-2)/k \rceil + 1}$. Because $j_x$ differs from $z^{\ast}_{\lceil (n-2)/k \rceil + 1}$ and satisfies $f(j_x) < f(z^{\ast}_{\lceil (n-2)/k \rceil + 1})$, by~\cref{prop:strict_min_increasing} it follows that $m(j_x) < m(z^{\ast}_{\lceil (n-2)/ k \rceil + 1})$.
Therefore, 
\begin{align*}
m(z_n) &= \sum_{i=1}^{k} m(j_i) \\
& = \sum_{i=1}^{x-1} m(z^{\ast}_{\lceil (n-2)/k \rceil + 1}) + \sum_{i=x}^{k} m(j_i) \\
& \leq \sum_{i=1}^{x-1} m(z^{\ast}_{\lceil (n-2)/k \rceil + 1}) + \sum_{i=x}^{k} m(j_x) \\ 
& \leq \sum_{i=1}^{x-1} m(z^{\ast}_{\lceil (n-2)/k \rceil + 1}) + \sum_{i=x}^{k} m(z^{\ast}_{\lceil (n-2)/k \rceil}) \\ 
&< \sum_{i=1}^{x-1} m(z^{\ast}_{\lceil (n-2)/k \rceil + 1}) + \sum_{i=x}^{n-2+k-k\lceil (n-2)/k \rceil} m(z^{\ast}_{\lceil (n-2)/k \rceil + 1}) + \sum_{i=n-1+k-k\lceil (n-2)/k \rceil}^{k} m(z^{\ast}_{\lceil (n-2)/k \rceil}) \\ 
&= \sum_{i=1}^{n-2+k-k\lceil (n-2)/k \rceil} m(z^{\ast}_{\lceil (n-2)/k \rceil + 1}) + \sum_{i=n-1+k-k\lceil (n-2)/k \rceil}^{k} m(z^{\ast}_{\lceil (n-2)/k \rceil}) \\ 
&= m(z_n^{\ast}) = (n-1)(k-1)+1.
\end{align*}
We have reached a contradiction of the requirement that $\sum_{i=1}^{k} m(j_i) = (n-1)(k-1) + 1$. 

With contradictions in both cases, we conclude that $K(z_n) \geq_D K(z_n^{\ast})$, so that $z_n^{\ast}$ is the strictly $k$-furcating tree of minimal rank.
\end{proof}

\section{Proof of~\cref{lemma:beta_decreasing}}
\label{sec:proof_beta_decreasing}

In this appendix, we prove that the value of $\beta_k$ defined in~\cref{subsec:strictly_k_asymptotics}, representing the base of a growth constant associated with the maximal rank for strictly $k$-furcating trees, strictly decreases with $k$ for $k \geq 2$. 

We make use of several logarithmic inequalities, stated as lemmas.
\begin{lemma} 
\label{lemma: log-ineq-1} 
For integers $k \geq 2$, 
\begin{equation}
\frac{\log [(k+1)!]}{k^4} - \frac{\log[(k+2)!]}{(k+1)^4} > \frac{\log(k!)}{k^3 (k-1)} - \frac{\log[(k+1)!]}{(k+1)^3 k}.
\label{eq:E1}
\end{equation}
\end{lemma}
\begin{proof}
Rearranging the statement, we must prove
\begin{align*}
& \log[(k+1)!] \left[\frac{1}{k^4} + \frac{1}{(k+1)^3 k}\right] > \frac{\log(k!)}{k^3(k-1)} + \frac{\log[(k+2)!]}{(k+1)^4},
\end{align*}
or equivalently,
\begin{align*}
(-5k^3-6k^2-4k-1)\log(k!) + (k^4 + 5k^3 + 6k^2 + 4k + 1)(k-1) \log(k+1) > k^4(k-1) \log(k+2).
\end{align*}
We use the upper bound $k! \leq ek(k/e)^k$, which holds for all $k \geq 1$~\cite[p.~90, result (q)]{inequality-dictionary}. Using this upper bound, for $k \geq 12$, we prove the stronger inequality 
\begin{equation*}
(-5k^3-6k^2-4k-1) \log [ek(k/e)^k] + (k^4 + 5k^3 + 6k^2 + 4k + 1)(k-1) \log(k+1) > k^4(k-1) \log(k+2),
\end{equation*}
or equivalently,
\begin{equation*}
(-5k^3-6k^2-4k-1)[-k+1+(k+1) \log k ] + (k^4 + 5k^3 + 6k^2 + 4k + 1)(k-1) \log(k+1) > k^4(k-1) \log(k+2).
\end{equation*}
The cases of $k=2, 3, 4, 5, 6, 7, 8, 9, 10$ and 11 can all be verified individually in~\cref{eq:E1}. 

For $k \geq 12$, it suffices to show that $H(k)>0$, where
\begin{align*}
H(k) & = (5k^4+k^3-2k^2-3k-1) + (-5k^4-11k^3-10k^2-5k-1) \log k \\
    & \quad + (k^5+4k^4+k^3-2k^2-3k-1) \log(k+1) + (-k^5+k^4) \log(k+2).
\end{align*}
Informally, as $H(k)$ grows large, the terms of order $\Theta(k^5 \log k)$ and $\Theta(k^4 \log k)$ cancel, so that $H(k)$ is positive for $k$ sufficiently large that the $5k^4$ term dominates. To make this argument rigorous, we use Napier's inequality~\cite[p.~220]{inequality-dictionary}, by which $\log(k+2) \leq \log(k+1) + \frac{1}{k+1} \leq \log(k+1) + \frac{1}{k}$ for all $k \geq 1$.

With Napier's inequality, we see that $H(k)$ satisfies
\begin{align}
H(k) & \geq (5k^4+k^3-2k^2-3k-1) + (-5k^4-11k^3-10k^2-5k-1)\log k \nonumber \\
     & \quad + (k^5+4k^4+k^3-2k^2-3k-1) \log(k+1) + (-k^5+k^4) \Big[\log(k+1) + \frac{1}{k} \Big] \nonumber \\
     & = (4k^4+2k^3-2k^2-3k-1) + (-5k^4-11k^3-10k^2-5k-1)\log k + (5k^4+k^3-2k^2-3k-1) \log(k+1) \nonumber \\
& \geq (4k^4+2k^3-2k^2-3k-1) + (-11k^3-10k^2-5k-1)\log k + (k^3-2k^2-3k-1) \log(k+1),
\label{eq:mess}
\end{align}
where the last step uses the fact that for all $k \geq 1$, $5k^4 \log(k+1) - 5k^4 \log(k) \geq 0$. 

For $k \geq 1$, the right-hand-side of \cref{eq:mess} is greater than or equal to $(4k^4 - 6k^2) + (-11k^3 - 16k^2) \log k  + (k^3 - 6k^2) \log k$, so that it suffices to show that this quantity is positive. We have 
\begin{align}
(4k^4 - 6k^2) + (-11k^3 - 16k^2) \log k  + (k^3 - 6k^2) \log k 
     & \geq k^2[(4k^2-6) + (-10k - 22) \log k]  \nonumber \\
     & \geq k^2(3k^2 -14k \log k) \nonumber \\
     & = k^3(3k - 14 \log k). \label{eq:314}
\end{align}
In this chain of inequalities, we have used that $4k^2-6 > 3k^2$ for $k \geq 3$ and $-22 \geq -4k$ for $k \geq 6$. But $3k \geq 14 \log k$ for all $k \geq 12$, a fact that can be proved by noting that for $f(x)=3x-14 \log x$, $f(12) > 0$ and $f'(x) > 0$ for all $x \geq 5$. Hence, the expression in \cref{eq:314} is positive, so that $H(k) > 0$ for $k \geq 12$. 
\end{proof}

\begin{lemma} 
\label{lemma: log-ineq-2}
Suppose $k \geq 2$ and $i \geq 2$ are integers, and $x > 0$. Then
\begin{align*}
\sum_{j=1}^{k} \log(x+j) &> \frac{ (k+1)^{-(i+1)} \log(x + k + 1) + [k^{-i} - (k+1)^{-i} ] \log x}{k^{-(i+1)} - (k+1)^{-(i+1)}} .
\end{align*}
\end{lemma}
\begin{proof}
We apply an integral lower-bound on the sum of logarithms to simplify the left-hand side. Because $\log x $ is monotonically increasing, for $x > 0$, $\log (x+1) \geq \int_{x}^{x+1} \log t \ dt$. Then
\begin{align*}
\sum_{j=1}^{k} \log(x+j) &\geq \int_{x}^{x + k} \log t\ dt = \left[t \log t - t \right] \Big|^{x+k}_{x} \\ 
&= (x+k) \log(x+k) - x\log x - k.
\end{align*}
It then suffices to demonstrate that for integers $k, i \geq 2$ and $x>0$, 
\begin{align*}
& (x+k) \log(x+k) - x\log x - k > \frac{(k+1)^{-(i+1)} \log(x + k + 1) + \big[k^{-i} - (k+1)^{-i} \big] \log x}{{k^{-(i+1)} - (k+1)^{-(i+1)}}}.
\end{align*}

Writing $F(x, k, i) = [(x+k) \log(x+k) - x\log x - k] - \big[(k+1)^{-(i+1)} \log(x + k + 1) + \big[k^{-i} - (k+1)^{-i}\big] \log x \big] /  [k^{-(i+1)} - (k+1)^{-(i+1)}]$, we must show that $F(x,k,i) > 0$. We complete the proof in two parts: (i) $\frac{\partial}{\partial i} F(x,k,i) \geq 0$, (ii) $F(x,k,2) > 0$. 

For (i), 
\begin{align*}
\frac{\partial F}{\partial i}(x,k,i) = \frac{k^{i+1} (k+1)^{i+1} [\log k - \log(k+1)] \, [\log x - \log(x+k+1)]}{[(k+1)^{i+1} - k^{i+1}]^2}. 
\end{align*}
Because $\log k - \log(k+1) < 0$ and $\log x - \log(x+k+1) < 0$, it follows that $\frac{\partial}{\partial i} F(x,k,i) > 0$ for all $i \geq 2$. 

For (ii), to show that $F(x,k,2) \geq 0$ for $x > 0$, we show that $\frac{\partial}{\partial x} F(x,k,2) < 0$ and $\lim_{x \to \infty} F(x,k,2) = 0$, so that as $x$ increases, $F$ monotonically decreases to 0. First,
\begin{align}
\frac{\partial F(x,k,2)}{\partial x} 
&= \log\left(1 + \frac{k}{x} \right) - \frac{1}{k^{-3} - (k+1)^{-3}} \left[\frac{(k+1)^{-3}}{x+k+1} + \frac{k^{-2} - (k+1)^{-2}}{x} \right].
\label{eq:partialx}
\end{align}
To prove that $\frac{\partial}{\partial x} F(x,k,2) < 0$ for all $x > 0$, we employ the same strategy and show that $\frac{\partial^2}{\partial x^2} F(x,k,2) > 0$ for all $x > 0$ and $\lim_{x \to \infty} \frac{\partial}{\partial x}F(x,k,2) = 0$. That $\lim_{x \to \infty} \frac{\partial}{\partial x}F(x,k,2) = 0$ follows quickly from \cref{eq:partialx}.

The second partial derivative of $F$ is
\begin{align*}
\frac{\partial^2 F(x,k,2)}{\partial x^2} = \frac{k^2[2k^4+k^3(3x+7)+k^2(x^2+8x+9)+k(x^2+7x+5)+(x+1)^2]}{x^2(3k^2+3k+1)(x+k)(x+k+1)^2},
\end{align*}
a positive quantity, as $k,x >0$ and all terms in the fraction are positive. 
Hence, $\frac{\partial}{\partial x} F(x,k,2) < 0$ for $x >0$. 

The last step is to show $\lim_{x \to \infty} F(x,k,2)=0$. We rewrite $F(x,k,2)$:
\begin{align*}
F(x,k,2) = \frac{k}{3k^2+3k+1} \Bigg[k^2 \log \bigg[ \frac{(x+k)^3}{x^2(x+k+1)} \bigg] + (3k+1) \log \bigg( 1 + \frac{k}{x} \bigg) \Bigg] + x \log\bigg(1+\frac{k}{x} \bigg) - k.
\end{align*}
The first term has limit 0, and the second has limit $k$. The limit of the sum of the three terms is 0.
\end{proof}

\begin{lemma} 
\label{lemma: log-ineq-3}
For integers $k \geq 2$ and real numbers $x, y > 0$ with $x > y$,
\begin{align*}
\log \left[\frac{\prod_{j=1}^{k} (x+j)}{x^k} \right] &> \log \left[\frac{\prod_{j=1}^{k} (y+j)}{y^k} \right].
\end{align*}
\end{lemma}
\begin{proof}
It suffices to show that $f(x) = x^{-k} \prod_{j=1}^{k} (x+j)$ decreases as $x$ increases. The derivative satisfies: 
\begin{align*}
f'(x) 
&= \frac{x^{k-1} \left[ \prod_{j=1}^{k} (x+j) \right] \left[ \left( \sum_{j=1}^{k} \frac{x}{x+j} \right) - k\right]}{x^{2k}} \\
& < \frac{x^{k-1} \left[ \prod_{j=1}^{k} (x+j) \right] \left[ \left( \sum_{j=1}^{k} 1 \right) - k\right]}{x^{2k}} = 0. 
\end{align*}
\end{proof}

\medskip
\noindent {\bf Proof of~\cref{lemma:beta_decreasing}.}
For $k \geq 2$, $\beta_k = \exp(z_3k^{-3}) \, \exp (\sum_{i=3}^{\infty} k^{-(i+1)} \rho_i )$. In this expression, $z_3$ depends on $k$, $z_3 = y_3 - \log(k!) / (k-1) = - \log(k!)/(k-1)$ because $y_3=0$; $\rho_i = \log \big[1 + P_k(d_i) / d_i^{k} \big]$, where $P_k(x)$ is the polynomial $(x+k)(x+k-1)\cdots (x+1) - x^k$; and $d_i = {d_{i-1} + k \choose k}$ for $i \geq 3$, with $d_3=1$. 

Hence, we must show 
\begin{align}
& \exp\bigg(-\frac{\log[(k+1)!]}{(k+1)^3 k}\bigg) \exp\bigg[\sum_{i=3}^{\infty} (k+1)^{-(i+1)} \log \bigg(1 + \frac{P_{k+1}(d_i')}{d_i'^{k+1}} \bigg) \bigg] \nonumber \\
& < \exp\bigg(-\frac{\log(k!)}{k^3(k-1)}\bigg) \exp \bigg[\sum_{i=3}^{\infty} k^{-(i+1)} \log \left(1 + \frac{P_{k}(d_i)}{d_i^k} \right) \bigg],
\label{eq:BkBkplus1}
\end{align} 
where we use $d_i'$ to denote the sequence $\{d_n\}_{n=3}^{\infty}$ generated with $(k+1)$-furcation and $d_i$ denotes the sequence $\{d_n\}_{n=3}^{\infty}$ obtained with $k$-furcation. 

Taking logarithms of both sides of \cref{eq:BkBkplus1} and rearranging, we must show
\begin{align}
\sum_{i=3}^{\infty} \bigg[k^{-(i+1)} \log \bigg(1 + \frac{P_{k}(d_i)}{d_i^k} \bigg)- (k+1)^{-(i+1)} \log \bigg(1 + \frac{P_{k+1}(d_i')}{d_i'^{k+1}} \bigg) \bigg] &> \frac{\log(k!)}{k^3 (k-1)} - \frac{\log[(k+1)!]}{(k+1)^3 k}.
\label{eq:appendixE}
\end{align}
It suffices to show (i) the $i=3$ term exceeds the right-hand side, and (ii) all terms with $i \geq 4$ are positive.

(i) If $i=3$, then $d_3 = d_3' = 1$, $P_k(d_3) = P_k(1) = (1+k) \big(1+(k-1)\big) \cdots (1+1) - 1^k = (k+1)! - 1$, and $P_{k+1}(1) = \big(1+(k+1)\big) (1+k) \cdots (1 + 1) - 1^{k+1} = (k+2)! - 1$. We claim that the $i=3$ term of \cref{eq:appendixE} is larger than the right-hand side, or 
\begin{align*}
k^{-4} \log \left[1 + \frac{(k+1)!-1}{1^k} \right] - (k+1)^{-4} \log \left[1 + \frac{(k+2)! - 1}{1^{k+1}} \right] &> \frac{\log(k!)}{k^3(k-1)} - \frac{\log [(k+1)!]}{(k+1)^3 k} 
\end{align*}
This inequality is true by Lemma~\ref{lemma: log-ineq-1}.

(ii) It remains to show that for $k$ fixed, for each $i \geq 4$,
\begin{equation*}
k^{-(i+1)} \log \bigg[1 + \frac{P_{k}(d_i)}{d_i^k} \bigg]- (k+1)^{-(i+1)} \log \bigg[1 + \frac{P_{k+1}(d_i')}{d_i'^{k+1}} \bigg] > 0. 
\end{equation*} 
This inequality can be rearranged as follows: 
\begin{align} 
\bigg( \frac{k}{k+1} \bigg)^{-(i+1)} & > \frac{\log \Big[1 + \frac{P_{k+1}(d_i')}{d_i'^{k+1}} \Big]}{\log \Big[1 + \frac{P_{k}(d_i)}{d_i^{k}} \Big]} 
= \frac{\log \Big[\prod_{j=1}^{k+1} (d_i' + j) \Big] - \log(d_i'^{k+1})}{\log \left[\prod_{j=1}^{k} (d_i + j) \right] - \log(d_i^{k})}, \nonumber \\
k^{-(i+1)} \bigg[ \bigg( \sum_{j=1}^{k} \log (d_i + j) \bigg) - \log(d_i^{k}) \bigg] & > (k+1)^{-(i+1)} \bigg[ \bigg( \sum_{j=1}^{k+1} \log (d_i' + j) \bigg) - \log(d_i'^{k+1}) \bigg].
\label{eq:G}
\end{align}

Define the bivariate function $G(x,k) = k^{-(i+1)} \big[\big(\sum_{j=1}^{k} \log(x+j) \big) - \log(x^k) \big]$ for $x > 0$, $k \geq 2$, and fixed $i \geq 4$. We claim (a) $G(x,k) > G(x,k+1)$, (b) $G(x,k) > G(y,k)$ if $x < y$, and (c) $d_i < d_i'$ for $i \geq 4$. We then have $G(d_i, k) > G(d_i, k+1)$ by (a); by (c), (b) applies, so that $G(d_i, k+1) > G(d_i', k+1)$, from which $G(d_i, k) > G(d_i', k+1)$, proving the inequality in \cref{eq:G}. It remains to show (a), (b), and (c).

(a) We show $G(x,k) > G(x,k+1)$ for $x > 0$ and integers $k \geq 2$. The desired inequality is equivalent to each of the following inequalities, the last of which holds by Lemma~\ref{lemma: log-ineq-2}.
\begin{align*}
k^{-(i+1)} \bigg[\bigg(\sum_{j=1}^{k} \log(x+j) \bigg) - k \log x \bigg] &> (k+1)^{-(i+1)} \bigg[ \bigg(\sum_{j=1}^{k+1} \log(x+j)  \bigg) - (k+1) \log x \bigg] \\ 
[k^{-(i+1)} - (k+1)^{-(i+1)}] \sum_{j=1}^{k} \log(x+j) &> (k+1)^{-(i+1)} \log(x+k+1) + [k^{-i} - (k+1)^{-i}] \log x \\ 
\sum_{j=1}^{k} \log(x+j) &> \frac{ (k+1)^{-(i+1)} \log(x + k + 1) + [k^{-i} - (k+1)^{-i}] \log x }{k^{-(i+1)} - (k+1)^{-(i+1)}} .
\end{align*}

(b) For $x > y > 0$, the desired inequality is equivalent to each of the following inequalities, the second of which holds by Lemma~\ref{lemma: log-ineq-3}.
\begin{align*}
k^{-(i+1)} \bigg[ \bigg( \sum_{j=1}^{k} \log(x+j) \bigg) - \log(x^k) \bigg] &> k^{-(i+1)} \bigg[ \bigg( \sum_{j=1}^{k} \log(y+j) \bigg) - \log(y^k) \bigg] \\ 
\log \bigg[\frac{\prod_{j=1}^{k} (x+j)}{x^k} \bigg] &> \log \bigg[\frac{\prod_{j=1}^{k} (y+j)}{y^k} \bigg].
\end{align*}

(c) To show that $d_i' > d_i$ for each $i \geq 4$, we proceed by induction on $i$. We have $d_3'=d_3=1$. For the base case of $i=4$, $d_{4}' = \binom{1 + (k+1)}{k+1} = k+2 > k + 1 = \binom{1 + k}{k} = d_4$. 

Assuming that $d_i' > d_i$ for each $i$ with $4 \leq i \leq n-1$, it follows that
\begin{align*}
d_n' = \binom{d_{n-1}' + k + 1}{k + 1} &= 
\frac{d_{n-1}' + k + 1}{k + 1} \cdot \frac{(d_{n-1}' + k)(d_{n-1}'+k-1) \cdots (d_{n-1}' + 1)}{k!} \\ 
&> \frac{d_{n-1}' + k + 1}{k + 1} \cdot \frac{(d_{n-1} + k)(d_{n-1}+k-1) \cdots (d_{n-1} + 1)}{k!} \\
&=  \frac{d_{n-1}' + k + 1}{k + 1} \cdot \binom{d_{n-1} + k}{k} > \binom{d_{n-1} + k}{k} = d_n.
\end{align*}
$\square$

\section{Proof of~\cref{thm:at_most_min_equals_ast}}
\label{pf:at_most_min_equals_ast}

\begin{proof}
We induct on $n$. The base case of $n=1$ is trivial because only one tree has with one leaf: $y_1 = y_1^{\ast}$. 

For the inductive hypothesis, suppose that $y_\ell = y_\ell^{\ast}$ for each $\ell$ with $1 \leq \ell \leq n-1$. For the inductive step, let $j_1, j_2, \ldots, j_k$ be the $k$ subtrees of $y_n$ in canonical order. Because $y_n$ is an at-most-$k$-furcating tree, some of these subtrees could have no leaves. 

We first show that $j_i$ is the minimal-rank tree for its number of leaves, or $j_i = y^{\ast}_{m(j_i)}$ for all $i$, $1 \leq i \leq k$. Suppose for contradiction that $j_i$ and $y^{\ast}_{m(j_i)}$ are different trees for some $i$, $1 \leq i \leq k$. Therefore, the tree $t_i'$ defined by replacing subtree $j_i$ by $y^{\ast}_{m(j_i)}$ is distinct from tree $y_n$. Moreover, because $m(j_i) < n$, the inductive hypothesis yields that $y^{\ast}_{m(j_i)} = y_{m(j_i)}$. Hence, $f(y^{\ast}_{m(j_i)})=f(y_{m(j_i)}) < f(j_i)$ because trees $y_{m(j_i)}$ and $j_i$ are distinct and have the same number of leaves. $K(t_i')$ is the $k$-tuple formed by replacing the $i$th entry of $K(y_n)$ by $f(y^{\ast}_{m(j_i)})$ and sorting entries in decreasing order. By Lemma~\ref{lemma: replacement-lemma}, $K(t_i') <_D K(y_n)$, so $f(t_i') < f(y_n)$. We have reached a contradiction of the rank-minimality of $y_n$. We can therefore conclude that $j_i=y^{\ast}_{m(j_i)}$ for all $i$, $1 \leq i \leq k$, and thus, $$K(y_n) = \Big(f(j_1), f(j_2), \ldots, f(j_k)\Big) = \Big(f(y^{\ast}_{m(j_1)}), f(y^{\ast}_{m(j_2)}), \ldots, f(y^{\ast}_{m(j_k)})\Big).$$

Note that $f(y^{\ast}_{m(j_1)}) \geq f(y^{\ast}_{m(j_2)}) \geq \ldots \geq f(y^{\ast}_{m(j_k)})$ because $j_1, j_2, \ldots, j_k$ are in canonical order, so we must have that $m(j_1) \geq m(j_2) \geq \ldots \geq m(j_k)$ by~\cref{prop:up_to_k_minimal_increasing}. Assume for the sake of contradiction that $y_n \neq y_{n}^{\ast}$ and let $x$ be the smallest index where $K(y_n)$ differs from $K(y_n^{\ast})$. We have two cases for the value of $x$. \hfill \break

\noindent 
\textit{Case (i): $x \leq n - k \lceil \frac{n}{k} \rceil$}. Because $K(y_n) <_D K(y^{\ast}_n)$ and $x \leq n - k \lceil n/k \rceil$, we must have that $f(y^{\ast}_{m(j_x)}) < f(y^{\ast}_{\lceil n/k \rceil})$. By~\cref{prop:up_to_k_minimal_increasing}, it follows that $m(j_x) < \lceil n/k \rceil$. Therefore,
\begin{align*}
m(y_n) &= \sum_{i=1}^{k} m(j_i) \\
& = \sum_{i=1}^{x-1} m(j_i) + \sum_{i=x}^{k} m(j_i) \\
& = \sum_{i=1}^{x-1} \bigg \lceil \frac{n}{k} \bigg \rceil + \sum_{i=x}^{k} m(j_i) \\
& \leq \sum_{i=1}^{x-1} \bigg \lceil \frac{n}{k} \bigg \rceil + \sum_{i=x}^{n - k \lceil \frac{n}{k} \rceil} m(j_x) + \sum_{i=n - k \lceil \frac{n}{k} \rceil + 1}^{k} m(j_i) \\ 
&< \sum_{i=1}^{x-1} \bigg \lceil \frac{n}{k} \bigg \rceil + \sum_{i=x}^{n - k \lceil \frac{n}{k} \rceil} \bigg\lceil \frac{n}{k} \bigg\rceil + \sum_{i=n - k \lceil \frac{n}{k} \rceil + 1}^{k} m(j_i) \\
& \leq \sum_{i=1}^{n - k \lceil\frac{n}{k} \rceil} \bigg\lceil \frac{n}{k} \bigg\rceil + \sum_{i=n - k \lceil\frac{n}{k} \rceil + 1}^{k} m(j_x) \\
& \leq \sum_{i=1}^{n - k \lceil\frac{n}{k} \rceil} \bigg\lceil \frac{n}{k} \bigg\rceil + \sum_{i=n - k \lceil\frac{n}{k} \rceil + 1}^{k} \bigg\lfloor \frac{n}{k} \bigg\rfloor \\ 
& = m(y_n^{\ast}) = n.
\end{align*}
We have reached a contradiction of the fact that $m(y_n) = n$.
\hfill \break

\noindent
\textit{Case (ii): $x > n - k \lceil \frac{n}{k} \rceil$}. Because $K(y_n) <_D K(y^{\ast}_n)$ and $x > n - k \lceil n/k \rceil$, we must have that $f(y^{\ast}_{m(j_x)}) < f(y^{\ast}_{\lfloor n/k \rfloor})$. By~\cref{prop:up_to_k_minimal_increasing}, it follows that $m(j_x) < \lfloor n/k \rfloor$. Hence, 
\begin{align*}
m(y_n) &= \sum_{i=1}^{k} m(j_i) \\
&= \sum_{i=1}^{n - k \lceil \frac{n}{k} \rceil} m(j_i) + \sum_{i=n - k \lceil \frac{n}{k} \rceil + 1}^{x-1} m(j_i) + \sum_{i=x}^{k} m(j_i) \\
&= \sum_{i=1}^{n - k \lceil \frac{n}{k} \rceil} \bigg\lceil \frac{n}{k} \bigg\rceil + \sum_{i=n - k \lceil \frac{n}{k} \rceil + 1}^{x-1} \bigg\lfloor \frac{n}{k} \bigg\rfloor + \sum_{i=x}^{k} m(j_i) \\
&\leq \sum_{i=1}^{n - k \lceil \frac{n}{k} \rceil} \bigg\lceil \frac{n}{k} \bigg\rceil + \sum_{i=n - k \lceil \frac{n}{k} \rceil + 1}^{x-1} \bigg\lfloor \frac{n}{k} \bigg\rfloor + \sum_{i=x}^{k} m(j_x) \\
& < \sum_{i=1}^{n - k \lceil \frac{n}{k} \rceil} \bigg\lceil \frac{n}{k} \bigg\rceil + \sum_{i=n - k \lceil \frac{n}{k} \rceil + 1}^{x-1} \bigg\lfloor \frac{n}{k} \bigg\rfloor + \sum_{i=x}^{k} \bigg\lfloor \frac{n}{k} \bigg\rfloor \\ &= \sum_{i=1}^{n - k \lceil \frac{n}{k} \rceil} \bigg\lceil \frac{n}{k} \bigg\rceil +\sum_{i=n - k \lceil \frac{n}{k} \rceil+1}^{k} \bigg\lfloor \frac{n}{k} \bigg\rfloor \\
&= m(y_n^{\ast}) = n,
\end{align*}
a contradiction of the fact that $m(y_n) = n$.

Because both cases have produced contradictions, we conclude that $y_n = y_n^{\ast}$
\end{proof}

\section{Proof of Theorem~\ref{thm:gamma-beta-comparison}}
\label{pf:gamma-beta-related}

We prove Theorem~\ref{thm:gamma-beta-comparison}, stating that the growth constant $\gamma_k$ in the at-most-$k$-furcating case is greater than or equal to the corresponding constant $\beta_k$ in the strictly $k$-furcating case, with equality if and only if $k=2$.

First, if $k=2$, then \cref{eq:recurrence_bn} and \cref{eq:recurrence_Bn} are the same recurrence, meaning that $b_n = B_n$ for all $n \geq 1$. By~\cref{thm:asymptotic_b_n} and~\cref{thm:asymptotic_B_n}, the growth constants of these identical recurrences are equal: $\gamma_k = \beta_k$. 

Next, consider $k \geq 3$. By~\cref{thm:asymptotic_b_n} and~\cref{thm:asymptotic_B_n}, 
\begin{align} 
\label{eq:gamma-beta-fraction-growth}
\frac{B_n}{b_n} \sim \frac{(k!)^{\frac{1}{k-1}} \gamma_k^{(k^n)}}{(k!)^{\frac{1}{k-1}} \beta_k^{(k^n)}} = \Big(\frac{\gamma_k}{\beta_k} \Big)^{(k^n)}.
\end{align}
The limit $\lim_{n \rightarrow \infty} (\gamma_k/\beta_k)^{(k^n)}$ is infinite if $\gamma_k > \beta_k$, 1 if $\gamma_k = \beta_k$, and $0$ if $\gamma_k < \beta_k$. By \cref{eq:gamma-beta-fraction-growth}, the limiting ratio of $B_n/b_n$ with $(\gamma_k/\beta_k)^{(k^n)}$ is 1. If $B_n/b_n \geq c$ for all $n \geq n_0$ and some constant $c > 1$, then the limit of $(\gamma_k/\beta_k)^{(k^n)}$ also exceeds 1, and hence it is infinite. Therefore, to show that $\gamma_k > \beta_k$, it suffices to show that $B_n/b_n \geq c$ for all $n \geq n_0$ and some $c > 1$. The folllowing lemma proves this statement, with $(c,n_0)=(2,4)$.

\begin{lemma} 
\label{lemma:bn-Bn-comparison}
Consider $k \geq 3$. For a constant $c=2$, $cb_n < B_n < b_{n+1}$ for all $n \geq 4$.
\end{lemma}
\begin{proof}
We first prove the right-hand side $B_{n} < b_{n+1}$, inducting on $n$ (and starting with $n=3$). For the base case of $n=3$, by \cref{eq:recurrence_Bn}, noting $B_2=2$, $B_3 = 2-B_2 + {B_2 + k - 1 \choose k} = 2-2+{2+k-1 \choose k} = k+1$.
By \cref{eq:recurrence_bn}, noting $b_3=3$, $b_4=2+{b_3+k-2 \choose k} = k+3 > B_3$.

Fix $n \geq 4$ and suppose $B_\ell < b_{\ell+1}$ for all $\ell$, $3 \leq \ell \leq n-1$. By~\cref{eq:recurrence_bn} and~\cref{eq:recurrence_Bn},
\begin{align*}
b_{n+1} - B_n &= 2 + \binom{b_n + k - 2}{k} - 2 + B_{n-1} - \binom{B_{n-1} + k-1}{k} \\
&= B_{n-1} + \binom{b_n + k - 2}{k} - \binom{B_{n-1} + k - 1}{k} \\ & \geq B_{n-1} \\ & > 0,
\end{align*}
where the first inequality holds by the inductive hypothesis $b_n \geq B_{n-1} + 1$ and because $\binom{x}{k}$ is increasing in $x$ for $x \geq k$. The induction is complete.

If $n=4$, we have that $cb_4 = c(k+3)$ and $B_4 = - (k - 1) + \binom{2k}{k}$. We claim that $c(x+3) < -(x - 1) + \binom{2x}{x}$ is true for all $x \geq 3$, and because $k \geq 3$, $cb_4 < B_4$. Indeed, for $x \geq 3$, 
\begin{align*}
-(x - 1) + \binom{2x}{x} \geq -(x - 1) + 2^{x-1} (x+1) \geq -(x - 1) + 4(x+1) = 3x+ 5 \geq c(x+3),
\end{align*}
where we use $\binom{2x}{x} \geq 2^{x-1} (x+1)$ for all $x \geq 1$ and $3x+5 \geq 2x+6$ for all $\tcr{x \geq 1}$.

Fix $n \geq 5$ and suppose $cb_\ell< B_\ell$ for all $\ell$, $4 \leq \ell \leq n-1$. By~\cref{eq:recurrence_bn} and~\cref{eq:recurrence_Bn},
\begin{align}
B_n - cb_n &= 2 - B_{n-1} + \binom{B_{n-1} + k - 1}{k} - c \bigg[2 + \binom{b_{n-1} + k - 2}{k} \bigg] \nonumber \\
&= 2(1-c) - B_{n-1} + \binom{B_{n-1} + k - 1}{k} - c \binom{b_{n-1} + k - 2}{k} \nonumber \\
&> 2(1-c) - b_n + \binom{B_{n-1}+k-1}{k} - c \binom{b_{n-1} + k - 2}{k} \nonumber \\
&= 2(1-c) - 2 - \binom{b_{n-1}+k-2}{k} + \binom{B_{n-1} + k - 1}{k} - c \binom{b_{n-1} + k - 2}{k} \nonumber \\
&= -2c + \binom{B_{n-1} + k - 1}{k} - (1 + c) \binom{b_{n-1} + k - 2}{k}. 
\label{eq:B_n-cb_n-simplified} 
\end{align}
In the inequality step, we have used $B_{n-1} < b_n$ for all $n \geq 4$.

We conduct a term-wise expansion of the expression in~\cref{eq:B_n-cb_n-simplified}, applying the inductive hypothesis: 
\begin{align}
& -2c + \frac{1}{k!} [(B_{n-1} + k - 1)(B_{n-1} + k - 2) \cdots (B_{n-1}+1)(B_{n-1}) \nonumber \\
& - (1+c)(b_{n-1} + k - 2)(b_{n-1} + k - 3)\cdots(b_{n-1})(b_{n-1} - 1)] \nonumber \\
&> -2c + \frac{1}{k!} [(cb_{n-1} + k - 1)(cb_{n-1} + k - 2) \cdots (cb_{n-1}+1)(cb_{n-1}) \nonumber \\
& \quad - (1+c)(b_{n-1} + k - 2)(b_{n-1} + k - 3)\cdots(b_{n-1})(b_{n-1} - 1)]. 
\label{ineq:binomial-expanded}
\end{align}
Noting that $2c+1 > 0$ and $c^2-c-1 > 0$, and $b_n > 0$ for all $n \geq 1$, trivially $b_{n-1} > -(2c+1)/(c^2-c-1)$ for $n \geq 5$. It then follows that $(cb_{n-1}+1)(cb_{n-1}) > (1+c)(b_{n-1})(b_{n-1}-1)$. 

Replacing the two last terms $(cb_{n-1}+1)(cb_{n-1})$ in the left-most product in~\cref{ineq:binomial-expanded} with $(1+c)(b_{n-1})(b_{n-1}-1)$, we have
\begin{align}
B_n - cb_n & > -2c + \frac{1}{k!} [(cb_{n-1} + k - 1)(cb_{n-1} + k - 2) \cdots (cb_{n-1}+2) \times (1+c)(b_{n-1})(b_{n-1}-1) \nonumber \\
& \qquad - (1+c)(b_{n-1} + k - 2)(b_{n-1} + k - 3)\cdots(b_{n-1})(b_{n-1} - 1)] \nonumber \\
& = -2c + \frac{1}{k!}(1+c)(b_{n-1})(b_{n-1}-1)\times \nonumber \\ 
& \qquad \Big[(cb_{n-1}+k-1)(cb_{n-1}+k-2) \cdots (cb_{n-1}+2) - (b_{n-1}+k-2)(b_{n-1}+k-3) \cdots (b_{n-1}+1) \Big] \nonumber \\
&> -2c + \frac{1}{k!}(1+c)(b_{n-1})(b_{n-1}-1) \times \nonumber \\
& \qquad \Big[(b_{n-1}+k-1)(b_{n-1}+k-2)\dots(b_{n-1}+2) - (b_{n-1}+k-2)(b_{n-1}+k-3) \cdots (b_{n-1}+1)\Big]. 
\label{ineq:binomial-expanded-ii}
\end{align}
We wish to show that the quantity in~\cref{ineq:binomial-expanded-ii} is positive for all $k \geq 3$ and $n \geq 5$. If $k=3$, noting that at $n=5$, $b_{5-1}=6$ (Table~\ref{tab:bn}), and $b_n$ is monotonically increasing in $n$ (\cref{prop:strict_monotonic}), the right-hand side of~\cref{ineq:binomial-expanded-ii} satisfies
\begin{align*}
-2c + \frac{1}{6} (1+c)(b_{n-1})(b_{n-1}-1) [(b_{n-1}+2) - (b_{n-1}+1)] \geq -2c + \frac{1}{6}(1+c)(b_4)(b_4-1) = 11 > 0.
\end{align*}
If $k \geq 4$, then recalling $b_4=k+3$, for $n\geq5$, the right-hand side of~\cref{ineq:binomial-expanded-ii} satisfies 
\begin{align*}
& -2c + \frac{1}{k!}(1+c)(b_{n-1})(b_{n-1}-1) \times (b_{n-1}+k-2)(b_{n-1}+k-3) \cdots(b_{n-1}+2) \times(k-2) \\
& \geq -2c + \frac{1}{k!}(1+c)(b_{4})(b_{4}-1) \times (b_4 + k-2) (b_4+k-3) \cdots(b_4+2) \times (k-2) \\ 
& = -2c + \frac{1}{k!}(1+c)(k+3)(k+2) \times (2k+1)(2k) \cdots(k+5) \times (k-2) \\
& = -2c + \frac{(2k+1)!}{(k+4)! \, k!} (1+c)(k+3)(k+2)(k-2) \\
& = -2c + {2k+1 \choose k}(1+c) \frac{k-2}{k+4} \\
& \geq -4 + {9 \choose 4}(3) \frac{2}{8}= \frac{181}{2} > 0.
\end{align*}
We conclude that $\gamma_k > \beta_k$ for all $k \geq 3$.
\end{proof}

\bibliographystyle{abbrv}
{\small \bibliography{sources.bib}}
\medskip

\end{document}